\documentclass[final,leqno,hidelinks]{siamltex}

\usepackage{graphicx,amsmath,amsfonts,latexsym,amssymb,hyperref,comment,subeqnarray}
\usepackage{hyperref,mathrsfs}
\usepackage[mathscr]{euscript}
\usepackage{colortbl}

\newtheorem{remark}{Remark}[section]
\def\mbx{\mathbf{x}}

\def\dx{\,\operatorname{dx}}
\def\dbx{\,\operatorname{dx}}
\def\dy{\,\operatorname{dy}}
\def\ds{\,\operatorname{ds}}
\def\O{\Omega}
\def\G{\Gamma}
\def\cI{\mathcal I}

\def\tS{{\widetilde S}}

\def\tO{\widetilde{\Omega}}
\def\tT{\widetilde{T}}

\def\cbe{\gamma}
\def\div{\nabla\cdot}
\def\grad{\nabla}
\def\eps{\epsilon}
\def\cE{\mathcal{E}_h}
\def\cEz{\mathring{\mathcal{E}}_h}
\def\Tau{\mathcal{T}}

\def\PHb{\widetilde{H}_\beta}

\def\tW{\widetilde{W}}

\def\PHz{\widetilde{H}_{\G,0}}
\def\tHi{\widetilde{H}_{\G}}

\def\la{\langle}

\def\ra{\rangle}

\newcommand{\<}{\left\langle}
\renewcommand{\>}{\right\rangle}
\def\lam{\lambda}
\def\Lam{\Lambda}
\def\nab{\nabla}

\def\bnu{{\boldsymbol\nu}}
\def\bn{\bnu}
\def\bal{{\boldsymbol\alpha}}
\renewcommand\O{\Omega}
\def\p{\partial}
\newcommand{\vertiii}[1]{{\left\vert\kern-0.25ex\left\vert\kern-0.25ex\left\vert
    #1\right\vert\kern-0.25ex\right\vert\kern-0.25ex\right\vert}}
\newcommand{\at}[1]{\kern-0.75ex\mid_{#1}}
\newcommand{\pp}[2]{\frac{\partial {#1}}{\partial {#2}}}
\newcommand{\jump}[2]{{\left[#1\right]}_{{#2}}}
\newcommand{\djump}[2]{{\left[\kern-0.25ex\left[#1\right]\kern-0.25ex\right]}_{{#2}}}

\newcommand{\cP}{\mathcal{P}}

\def\Span{\mathrm{\,Span\,}}

\def\forany{\quad\forall}

\newcommand{\g}{\gamma}
\numberwithin{equation}{section}

\DeclareFontEncoding{FMS}{}{}
\DeclareFontSubstitution{FMS}{futm}{m}{n}
\DeclareFontEncoding{FMX}{}{}
\DeclareFontSubstitution{FMX}{futm}{m}{n}
\DeclareSymbolFont{fouriersymbols}{FMS}{futm}{m}{n}
\DeclareSymbolFont{fourierlargesymbols}{FMX}{futm}{m}{n}
\DeclareMathDelimiter{\tbar}{\mathord}{fouriersymbols}{152}{fourierlargesymbols}{147}

\title{A Nonconforming Immersed Finite Element Methods for Elliptic Interface Problems}

\author{
Tao Lin
\thanks{Department of Mathematics, Virginia Tech, Blacksburg, VA 24061, tlin@math.vt.edu}
\and
Dongwoo Sheen\thanks{Department of Mathematics, and Interdisciplinary
  Program in Computational Science \& Technology, Seoul National University, Seoul 08826, Korea, dongwoosheen@gmail.com}
\and
Xu Zhang\thanks{Department of Mathematics and Statistics, Mississippi State University, Mississippi State, MS 39762, xuzhang@math.msstate.edu}
}

\begin{document}
\thanks{This research was partially supported by the National Science
  Foundation Grant (DMS-1016313, DMS-1720425), NRF--2017R1A2B3012506 and NRF-2015M3C4A7065662 in part.
}
\maketitle

\begin{abstract}
A new immersed finite element (IFE) method is developed for
second--order elliptic problems with discontinuous diffusion coefficient.
The IFE space is constructed based on the rotated--$Q_1$ nonconforming
finite elements with the integral-value degrees of freedom. 
The standard nonconforming Galerkin method is employed in this IFE method without any penalty stabilization term.
Error estimates in energy-- and $L^2$--norms are proved to be better than
$O(h\sqrt{|\log h|})$ and $O(h^2|\log h|)$, respectively, where
the logarithm factors reflect jump discontinuity.
Numerical results are reported to confirm our analysis.

\end{abstract}



\pagestyle{myheadings}
\thispagestyle{plain}

\section{Introduction}
We consider the second-order elliptic interface problem:
\begin{subeqnarray}\label{eq:BVP}
  -\nabla\cdot (\beta \nabla u)  &=& f~~~  \mbox{in}~~ \O^-\cup \O^+, \slabel{eq: elliptic PDE} \\
  u  &=& g~~~
  \mbox{on}~~\partial \O, \slabel{eq: elliptic BC}
\end{subeqnarray}
where, without loss of generality, we assume that a
$C^2$-continuous interface curve $\G$
separates the physical domain $\O$ into two sub-domains $\O^+$ and
$\O^-$, such that $\overline{\O} = \overline{\O^+\cup\O^-\cup\G}$, see an illustration in Figure \ref{fig: domain}.
The physical domain $\O\subset \mathbb{R}^2$ is assumed to be occupied by two materials such that the diffusion coefficient $\beta(x,y)$ is discontinuous across the interface $\G$, and it is assumed to be
a piecewise constant function defined by
\begin{equation}\label{eq: beta discontinuity}
\beta(x,y) = \left\{\begin{array}{lll}
\beta^-\quad \mbox{if}~~(x,y)\in \O^-,\\
\beta^+\quad \mbox{if}~~(x,y)\in \O^+,
\end{array}\right.
\end{equation}
such that $\min\{\beta^-, \beta^+\} > 0$. Across the interface $\G$, the solution and the normal component of the flux are assumed to be continuous, \emph{i.e.},
\begin{subeqnarray} \label{eq:jump_conditions}
\jump{u}{\G} &=& 0,\slabel{eq: elliptic jump}\\
\djump{\bn\cdot\beta \nabla u}{\G} &=& 0,\slabel{eq: elliptic flux jump}
\end{subeqnarray}
where $[v]_\G = v^+|_\G - v^-|_\G$, and $\djump{\bnu\cdot\beta\grad u}{\Gamma}
= \bnu^+\cdot\beta^+\grad u^+
+ \bnu^-\cdot\beta^-\grad u^-$, with ${\bn}$ the unit normal of $\Gamma$.

\begin{figure}[htb]
\centerline{
\hbox{\includegraphics[width=0.32\textwidth]{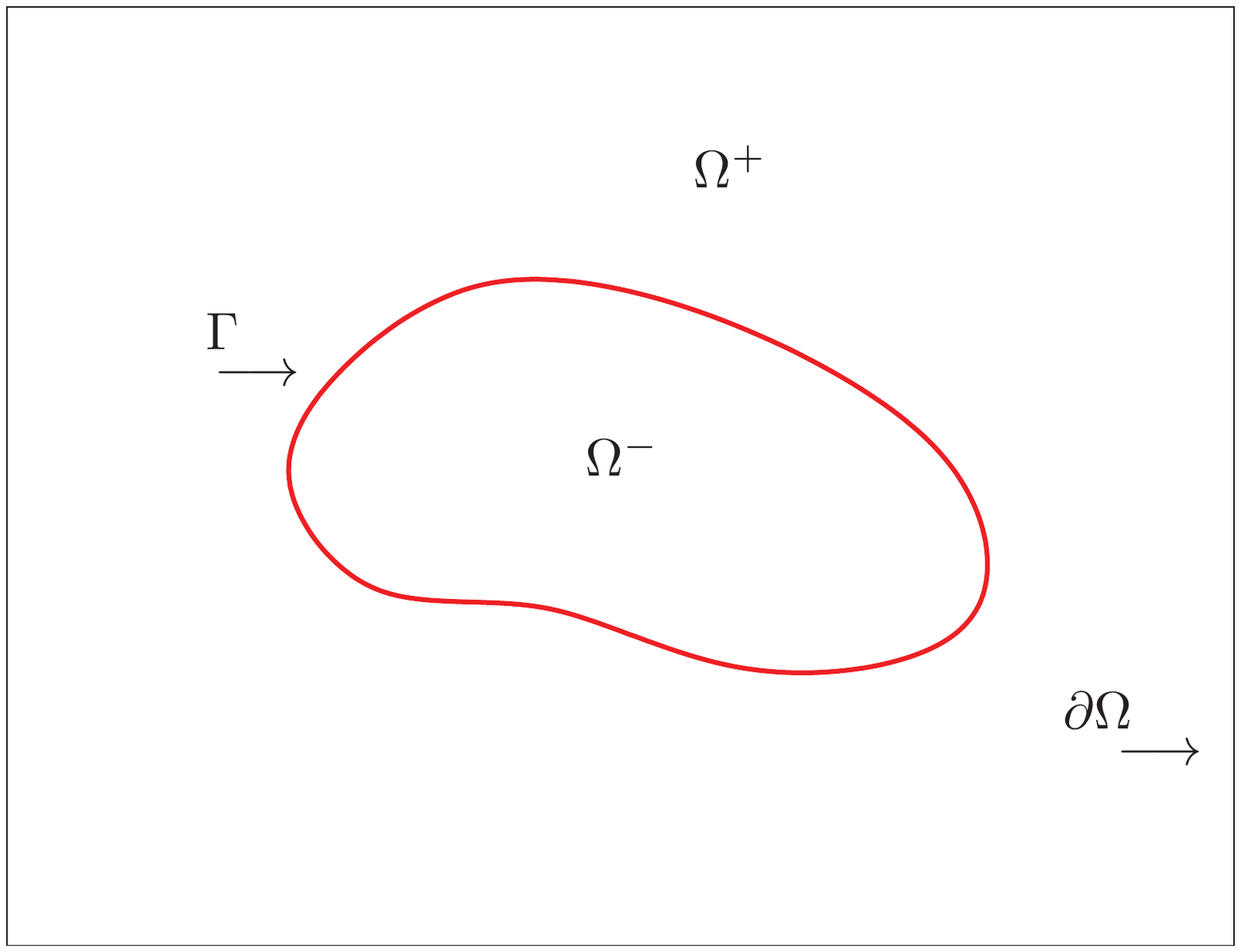}}
\hbox{\includegraphics[width=0.32\textwidth]{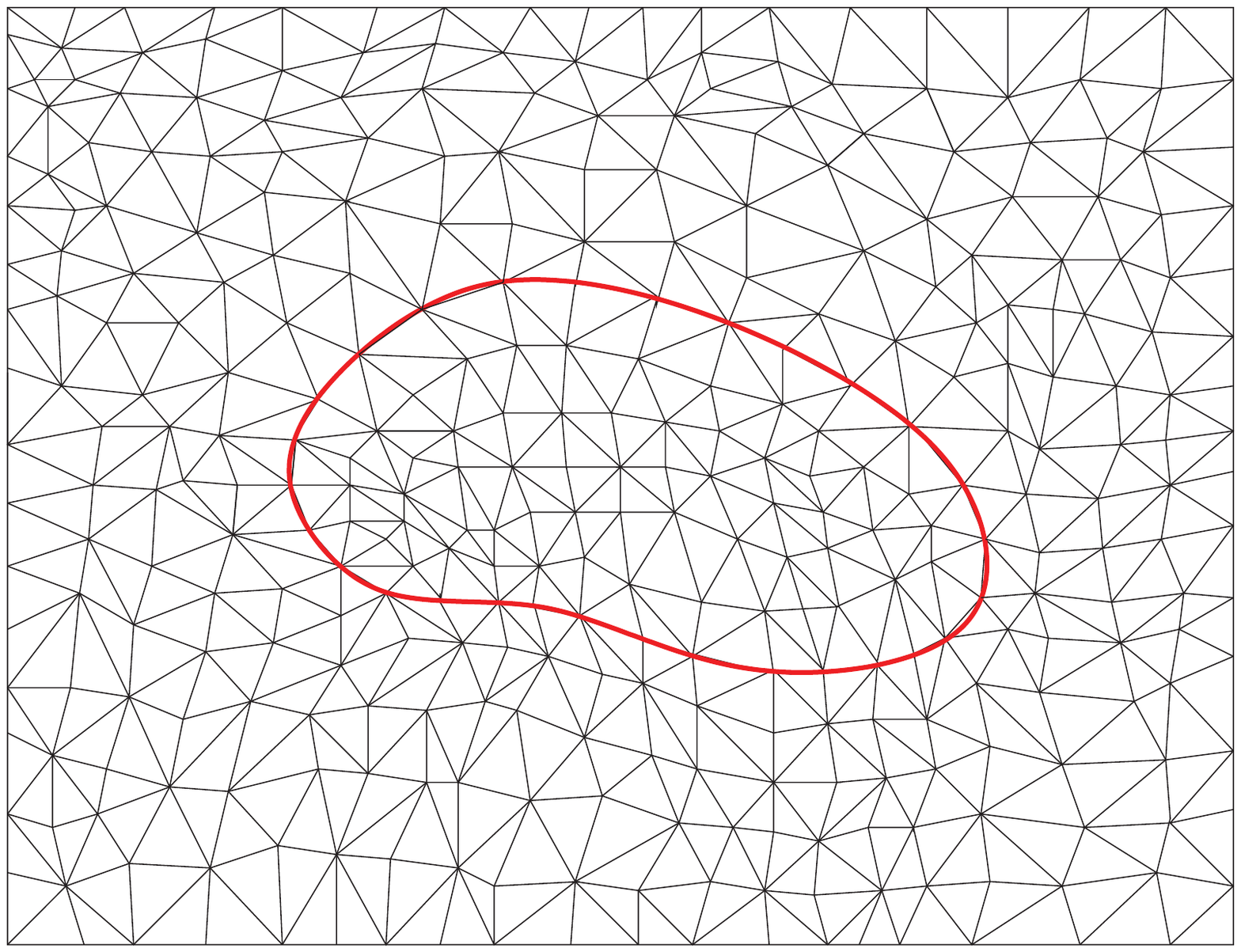}}
\hbox{\includegraphics[width=0.32\textwidth]{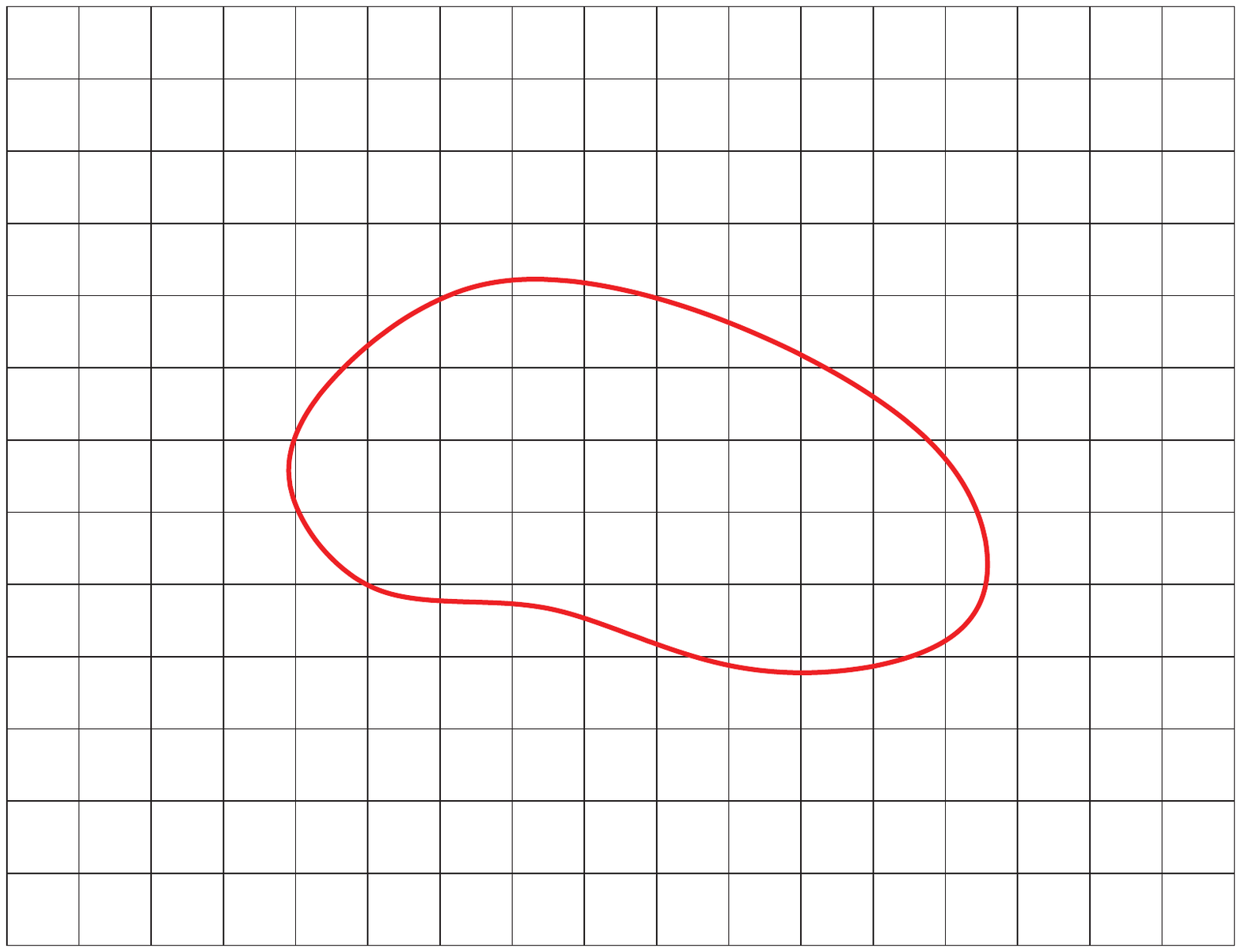}}
}
\caption{A geometry of interface problems. Body-fitting mesh and non-body-fitting mesh.}
\centering
\label{fig: domain}
\end{figure}
Conventional finite element methods (FEM) can solve this elliptic interface problem satisfactorily provided that solution meshes are shaped to fit the material interface \cite{1998ChenZou}; otherwise the accuracy of the solution is uncertain \cite{1970Babuska}. \emph{Immersed finite element} (IFE) methods \cite{2017CaoZhangZhang, 2010ChouKwakWee, 2007GongLiLi, 2017GuzmanSanchezSarkis, 2014JiChenLi, 2004LiLinLinRogers, 2003LiLinWu,2013LinSheenZhang, 2015LinYangZhang1, 2010VallaghePapadopoulo}, on the other hand, do not require meshes to fit the interface. Hence, if desired, Cartesian meshes can be used to solve interface problems which is advantageous in many simulations. For example, in particle-in-cell methods for plasma particle simulations \cite{2005KafafyLinLinWang, 2007KafafyWang}, it is preferable to solve the governing electric potential interface problem on Cartesian meshes for efficient particle tracking. Also, IFE methods, in either a standard fully discrete or a semi discrete (the method of lines) formulation,  can be used to solve time dependent problems with moving interfaces \cite{2013HeLinLinZhang,  2013LinLinZhang1} on a fixed Cartesian mesh throughout the whole simulation.

The basic idea of IFE methods is to locally modify finite element functions on interface elements to fit the interface jump conditions \eqref{eq: elliptic jump} and \eqref{eq: elliptic flux jump}.
For elliptic interface problems, most IFE methods in the literature are modified from the Lagrange type conforming finite element spaces, whose degrees of freedom are determined by nodal values at the mesh points. However, IFE spaces originated from these conforming FE spaces are usually nonconforming because IFE functions are discontinuous across interface edges. This discontinuity can be harmfully large for certain configuration of interface location and diffusion coefficient. Consequently, the IFE solution is often less accurate around the interface than the rest of solution domain. Our recent study in \cite{2015LinLinZhang, 2013ZhangTHESIS} indicates that the convergence rates of these conforming type IFE functions used in the Galerkin formulation can deteriorate as the mesh size gets small.

New \textit{partially penalized immersed finite element} (PPIFE) methods are recently introduced in \cite{2015LinLinZhang, 2013ZhangTHESIS} to cope with the negative impacts caused by the discontinuity in IFE functions. In the new PPIFE scheme, the flux jump terms and penalty terms are added on the interface edges to ensure the consistency and the stability of the scheme. These PPIFE methods significantly improve the numerical approximation as the errors around the interface are reduced dramatically, and overall convergence rates are very stable in both energy and the $L^2$- norms. Moreover, it is theoretically proved that the PPIFE methods convergence optimally in energy norm provided that the exact solution is piecewise smooth \cite{2015LinLinZhang}.

In this article, we develop a new IFE method that uses an alternative approach to effectively alleviate the harmful impacts of discontinuity in IFE functions. In this new framework, the continuity of each IFE function across element boundary is weakly enforced; hence there is \textit{no need to add penalty terms in the new scheme}. Specifically, the new IFE space is based on nonconforming finite element spaces \cite{1998ChenOswald, 1973CrouzeixRaviart,1999DouglasSantosSheenYe, 1996KloucekLiLuskin, 1992RanacherTurek} rather than conforming Lagrange type finite element spaces. One significant difference between conforming and
nonconforming finite elements is the way to impose the continuity of
finite element functions across elements. Conforming FE functions
enforce continuity through nodal values at mesh points, while the continuity of
nonconforming FE functions is imposed weakly through average values over edges. For IFE methods, an interface edge is cut by the interface into two pieces. The restriction of an IFE function to an
interface edge leads to two piecewise polynomials from the two
adjacent interface elements sharing this edge. In the conforming finite element 
framework, these two piecewise polynomials coincide at the two
endpoints of the interface edge, and this guarantees the continuity of
the IFE function at the end nodes of the interface edge but not
the whole interface edge. On the other hand, in the nonconforming framework, the continuity across an element edge is weakly enforced over the whole edge in the integral sense, no matter whether it is a polynomial or a piecewise polynomial. Thus we can take advantage of this nonconforming mechanism for constructing IFE functions that weakly preserve continuity over the whole of each interface edge.  

The simplest nonconforming finite element defined on triangular meshes is the well-known Crouzeix-Raviart element \cite{1973CrouzeixRaviart}. For rectangular meshes, the simplest nonconforming finite elements are known as the \textit{rotated-$Q_1$ finite elements} \cite{1998ChenOswald, 1999DouglasSantosSheenYe, 1996KloucekLiLuskin, 1992RanacherTurek}. Their degrees of freedom are determined either by values at midpoints of edges or by the integral values over edges. 
These two types of degrees of freedom define two different finite element spaces if the basis functions are taken as in \cite{1992RanacherTurek}, but an identical space if the basis functions are taken as in \cite{1998ChenOswald,2013JeonNamSheenShim}.
The Crouzeix-Raviart type IFE method was discussed in \cite{2010KwakWeeChang}. In this article, we develop a new IFE space based on the rotated-$Q_1$ functions with the integral-value degrees of freedom. We will derive quasi-optimal error estimates in both energy and $L^2$- norms for the simple Galerkin approximation. In our error analysis, we extend a special projection operator introduced in 
\cite{1999DouglasSantosSheenYe} to bound the flux error on edges. We show that the flux error on interface edges will have a $\log|h|$ factor. The techniques in our error analysis are new to interface problems and they are different from analysis in literature such as \cite{2010KwakWeeChang}.

The rest of this article is organized as follows. In Section 2, we present nonconforming rotated-$Q_1$ IFE space and present some basic properties. In Section 3, we discuss the approximation capabilities of the IFE space. In Section 4, we analyze errors of Galerkin solutions to the elliptic interface problem in energy and $L^2$ norms. In Section 5, numerical results are presented to confirm our analysis and to demonstrate features of the new IFE method. Finally, a few brief conclusions are provided in Section 6.


\section{Nonconforming Immersed Finite Element Space}
This section starts with notations and some preliminaries to be used in this paper. Then, we will introduce the IFE space based on nonconforming rotated-$Q_1$ elements.

\subsection{Notations and Preliminaries}
Multi-index notations will be employed such that $\alpha =
  (\alpha_1,\alpha_2)\in \left[\mathbb{Z}^+\right]^2, |\alpha| =
\alpha_1 + \alpha_2,$ together with the partial differential operator
$\p^\alpha = \frac{\p^{\alpha_1}}{\p x_1^{\alpha_1}}\frac{\p^{\alpha_2}}{\p
  x_2^{\alpha_2}}$, where $\mathbb{Z}^+$ denotes the set of all nonnegative integers.
By $\tS$ we denote the union
  of finite number of mutually disjoint open sets
  $S_j\subset \mathbb{R}^2,j=1,\cdots,J$, and by $S$ the interior of
  $\overline{\tS}$, which contains $\tS$ and all its possible interfaces.
If $J=1$, $\tS = S.$
Let $W^{m,p}(\tS)$ denote the usual Sobolev space with non-negative integer
index $m$, equipped the norm and seminorm:
\begin{equation*}
  \|v\|_{m,p,\tS} =\Big(\underset{|\alpha|\leq m}{\sum} \int_\tS\left|\p^\alpha v(\mbx)\right|^p \dbx \Big)^{1/p},~~
  |v|_{m,p,\tS}=\Big(\underset{|\alpha|= m}{\sum} \int_\tS\left|\p^\alpha v(\mbx)\right|^p \dbx \Big)^{1/p},
\end{equation*}
for $1\leq p<\infty$, and
\begin{equation*}
  \|v\|_{m,\infty,\tS} =\underset{|\alpha|\leq m}{\max} \text{ess.sup}\{|v(x)|:~x\in \tS\},~~
  |v|_{m,\infty,\tS}=\underset{|\alpha|= m}{\max} \text{ess.sup}\{|v(x)|:~x\in \tS\}.
\end{equation*}
In particular, for $p = 2$, we denote $H^m(\tS) = W^{m,p}(\tS)$, and we omit the index $p$
in associated norms and seminorms for simplicity, \emph{i.e.},
$\|v\|_{m,2,\tS} =\|v\|_{m,\tS}$, and $|v|_{m,2,\tS} = |v|_{m,\tS}$.
We will also follow the convention to drop the domain index $\tS$ if $\tS=\O.$
For $p=2$, associated with the norm $\|\cdot\|_{m,\tS}$, the inner product
for $H^m(\tS)$ will be denoted by $(\cdot,\cdot)_{H^m(\tS)}$, with further
simplification to $(\cdot,\cdot)_{\tS}$ and $(\cdot,\cdot)$ if $m=0$ and also if
$\tS=\O$, respectively.

For $m\ge 1$, we define two types of subspaces of $H^m(\tS)$ whose functions satisfy the
interface jump conditions \eqref{eq: elliptic jump}
and \eqref{eq: elliptic flux jump}
on $\Gamma$. First, we set
\[
\tHi^m(S) = H^1(S) \cap H^m(\tS),
\]
endowed with the inner-product and the norm
\[
\<u, v\>_{\tHi^m(S)} = (u,v)_{H^1(S)} + \sum_{j=1}^J \sum_{|\bal|=2}^m(\p^\bal u, \p^\bal
v)_{L^2(S^j)},\quad
\|u\|_{\tHi^m(S)} = \sqrt{\<u, u\>_{\tHi^m(S)}}.
\]
Notice that $\tHi^1(S) = H^1(S)$ and that
\[
\jump{v}{\G} = 0
\text{ in the sense of } H^{\frac{1}{2}}(\G)~~~\forall v\in \tHi^m(S),~m\ge 1.
\]

Finally, for $m= 2$, we define a subspace of $\tHi^m(S)$, which will
be suitable for the analysis of interface problem, as follows:
\begin{eqnarray*}
  \PHb^2(S) = \left\{v\in \tHi^2(S):~\djump{\bn_{\G} \cdot \beta \nabla v}{\G} = 0\right\}.
\end{eqnarray*}
In addition, the following spaces will be useful: for $p\ge 2$,
\begin{equation*}
\widetilde W^{2,p}_\G(S)=W^{1,p}(S) \cap W^{2,p}(\tS);\quad
{\tW}^{2,p}_\beta(S) = \{v\in \widetilde W^{2,p}_\beta(\Omega) \mid~
  \djump{\bn_{\G} \cdot \beta \nabla v}{\G} = 0\}.
\end{equation*}
Here, and in what follows, $\jump{v}{\G}$ and $\djump{v}{\G}$ will mean
the jumps across $\G$ in the sense of $W^{1-\frac1{p},p}(\G)$ and $W^{-\frac1{p},p}(\G)$,
respectively. However, if $u\in{W}^{2,p}(\tS)$, these jumps are
defined in the sense of
$W^{2-\frac1{p},p}(\G)$ and $W^{1-\frac1{p},p}(\G)$, respectively.


Assume that $f\in H^{-1}(\O),$ where $H^{-1}(\O)$ is the dual space of
$H_0^1(\O)=\PHz^1(\tO).$ For the interface problem described by \eqref{eq:BVP}
and \eqref{eq:jump_conditions}, we consider its weak form:
find $u\in H^1(\O)$ such that $u = g$ on $\partial \Omega$ and
\begin{equation}\label{eq:weakform}
    a(u,v) = L(v)~~~\forall~v\in H_0^1(\O),
\end{equation}
where
\begin{eqnarray*}
  a(u,v) = (\beta\grad u, \grad v), \quad
  L(v) = \<f, v\>_{H^{-1}(\O),H_0^1(\O)}, 
\end{eqnarray*}
$\<\cdot, \cdot\>_{V',V}$ being the duality pairing between the topological
vector space $V$ and its dual space $V'$.
An application of the Lax-Milgram Lemma shows that there exists a unique
solution $u\in H^1(\O)$ for \eqref{eq:weakform} such that
\[
\| u \|_1 \le C \| f\|_{-1},
\]
where $C$ is a positive constant depending only on $\O$ and $\beta.$

\subsection{Nonconforming FE functions}
Let $\O$ be a rectangular domain or a union of rectangular domains.
Without loss of generality, assume that $\{\mathcal{T}_h\}$ is a family of uniform Cartesian meshes for domain $\O$ with 
mesh parameter $h>0$. For each element $T\in\mathcal{T}_h$, we call it
an interface element if the interior of $T$ intersects with the interface
$\G$; otherwise, we call it a non-interface element. Without loss of generality, we assume that interface elements in $\mathcal{T}_h$ satisfy the following hypotheses when the mesh size $h$ is small enough:
\begin{description}
  \item \textbf{(H1)} The interface $\G$ cannot intersect an edge of any rectangular element at more than two points unless the edge is part of $\G$.
  \item \textbf{(H2)} The interface $\G$ can only intersect the boundary of an interface element at two points, and these intersection points must be on different edges of this element. 
\end{description}

Denote by $\mathcal{T}_h^i$ and $\mathcal{T}_h^n =
  \mathcal{T}_h\setminus\mathcal{T}_h^i$ the collections of all
  interface elements and non-interface elements, respectively.
For a typical rectangular element $T = \square A_1A_2A_3A_4 \in \mathcal{T}_h$, the following conventions for its vertices and edges are assumed:
\begin{equation}\label{eq: rectangle vertices}
  A_1 = (x_0,y_0),~~A_2 = (x_0+h_x,y_0),~~ A_3 = (x_0+h_x,y_0+h_y),~~A_4 (x_0,y_0+h_y),
\end{equation}
and
\begin{equation}\label{eq: rectangle edges}
    \cbe_1 = \overline{A_1A_2},~~\cbe_2 = \overline{A_2A_3},~~\cbe_3 = \overline{A_3A_4},~~\cbe_4 = \overline{A_4A_1}.
\end{equation}
We follow the classical triplet definition of a finite element \cite{1978Ciarlet}. On the element $T$, the local FE space is defined by
\begin{equation}\label{eq: pointwise DoF}
    \Pi_T = \text{Span}\left\{1,\frac{x-x_0}{h_x},\frac{y-y_0}{h_y},\left(\frac{x-x_0}{h_x}\right)^2-\left(\frac{y-y_0}{h_y}\right)^2\right\}.
\end{equation}
The degrees of freedom are defined as the average values over edges: 
\begin{equation}\label{eq: integral DoF}
    \Sigma_T = \left\{\frac{1}{|\cbe_j|}\int_{\cbe_j}\psi_T\,\mbox{d}s, j=1,2,3,4 :\forall ~\psi\in \Pi_T\right\},
\end{equation}
where $|\cbe_j|$ denotes the length of the edge $\cbe_j$.
The local basis functions $\psi_{j,T}$, $j=1,2,3,4$, fulfill
\begin{equation}\label{eq: integral constraints}
    \frac{1}{|\cbe_k|}\int_{\cbe_k}\psi_{j,T}\,\mbox{d}s = \delta_{jk},~~~\forall~ j,k = 1,2,3,4.
\end{equation}
Set the local finite element space on an element $T$ as follows
\begin{equation}\label{eq: noninterface local space P I}
  S_h^{n}(T) = \Span\{\psi_{j,T}:j = 1,2,3,4\}.
\end{equation}
It is obvious that on every element $T \in \mathcal{T}_h$, $S_h^{n}(T) = \Pi_T$.

\subsection{Nonconforming IFE Functions}
\label{sec: nonconforming IFE functions}

Next, we describe the construction of a local IFE function $\phi_T$ on a typical interface element $T\in\mathcal{T}_h^i$ whose vertices and edges are given in \eqref{eq: rectangle vertices} - \eqref{eq: rectangle edges}. 

Assume that an interface curve $\G$ intersects $T \in \mathcal{T}_h^i$ at two different points
$D$ and $E$, and the line segment $\overline{DE}$ separates $T$ into two subelements
$T^+$ and $T^-$.
Depending on the adjacency of the edges containing $D$ and $E$,
the interface elements will be classified as Type I and Type II interface elements
such that these two edges are located at two adjacent edges and at two opposite edges, respectively.
We use a Type II interface element to exemplify the construction of the local IFE functions and corresponding spaces, i.e., we assume the interface points are such that
\begin{equation*}
    D = (x_0+dh_x,y_0),~~ E = (x_0+eh_x,y_0+h_y),
\end{equation*}
where $d,e \in(0,1)$. The local IFE function $\phi_{T}$ is defined as a piecewise rotated $Q_1$ polynomial as follows:
\begin{equation}\label{eq: nonconforming IFE basis T2}
    \phi_{T}(x,y) =
    \left\{
      \begin{array}{ll} c_1^+ + c_2^+ \left(\dfrac{x-x_0}{h_x}\right) + c_3^+ \left(\dfrac{y-y_0}{h_y}\right) + c_4^+\left(\left(\dfrac{x-x_0}{h_x}\right)^2-\left(\dfrac{y-y_0}{h_y}\right)^2\right) & \text{in}~ T^+, \vspace{1mm} \\
       c_1^- + c_2^- \left(\dfrac{x-x_0}{h_x}\right) + c_3^- \left(\dfrac{y-y_0}{h_y}\right) + c_4^-\left(\left(\dfrac{x-x_0}{h_x}\right)^2-\left(\dfrac{y-y_0}{h_y}\right)^2\right) & \text{in}~ T^-. \\
      \end{array}
    \right.
\end{equation}
The coefficients $c_j^\pm$ are determined by the average value $v_j$ on each edge $\cbe_j$:
\begin{equation}\label{eq: nonconforming IFE basis partial nodal}
    \frac{1}{|\cbe_j|}\int_{\cbe_j} \phi_T\,\mbox{d}s = v_j,~~ j=1,2,3,4,
\end{equation}
and the following interface jump conditions 
\begin{equation}\label{eq: nonconforming IFE basis DE continuity}
\jump{{\phi_T}}{\overline{DE}} =0,
\end{equation}
and 
\begin{equation}\label{eq: nonconforming IFE basis flux continuity}
    \int_{\overline{DE}}\djump{{\bnu}_{\overline{DE}}\cdot  \beta\nabla\phi_{T}}{\overline{DE}} \mbox{d}s= 0,
\end{equation}
where ${\bnu}_{\overline{DE}}$ is the unit normal on $\overline{DE}$. Note that the continuity condition \eqref{eq: nonconforming IFE basis DE continuity} is equivalent to
\begin{equation}
    [\phi_{T}(D)] = 0,~~~~[\phi_{T}(E)] = 0,~~~~  c_4^+ = c_4^-.
\end{equation}

Equations \eqref{eq: nonconforming IFE basis partial nodal}--\eqref{eq: nonconforming IFE basis flux continuity} provide eight constraints and lead to an $8\times 8$ algebraic system $M_c{\bf C}={\bf V}$ about the coefficients ${\bf C} = (c_1^-,\cdots,c_4^-,c_1^+,\cdots,c_4^+)^t$ with ${\bf V} = (v_1,\cdots,v_4, 0,\cdots,0)^t$.
By direct calculations, one can verify that the matrix $M_c$ is nonsingular for all $\beta^\pm>0$ and $0<d,e<1$; see \cite{2013ZhangTHESIS} for more details.
Hence, an IFE function $\phi_{T}$ satisfying jump conditions \eqref{eq: nonconforming IFE basis DE continuity} and
\eqref{eq: nonconforming IFE basis flux continuity} is uniquely
determined by its integral values $v_j$ over edges $\cbe_j$,
$j=1,2,3,4$. For each $j=1,2,3,4$, let $\mathbf{V}=\mathbf{V}_j =
(v_1,\cdots,v_4,0,\cdots,0)^t\in \mathbb{R}^8$ be the $j$-th
canonical vector such that $v_j = 1$ and $v_k = 0$ for $k\neq j$.
We can solve for $\mathbf{C}_j =
(c_1^-,\cdots,c_4^-,c_1^+,\cdots,c_4^+)^t$ and use it in \eqref{eq:
  nonconforming IFE basis T2} to form the $j$-th nonconforming
rotated-$Q_1$ local IFE basis function $\phi_{j,T}$.
Figure \ref{fig: nonconforming FE IFE local basis integral} presents illustrations for
a comparison of a standard rotated $Q_1$ finite element basis function and its corresponding rotated $Q_1$ IFE basis functions in
both Type I and Type II interface elements.

\begin{figure}[!htb]
  \centering
  \includegraphics[width=0.3\textwidth]{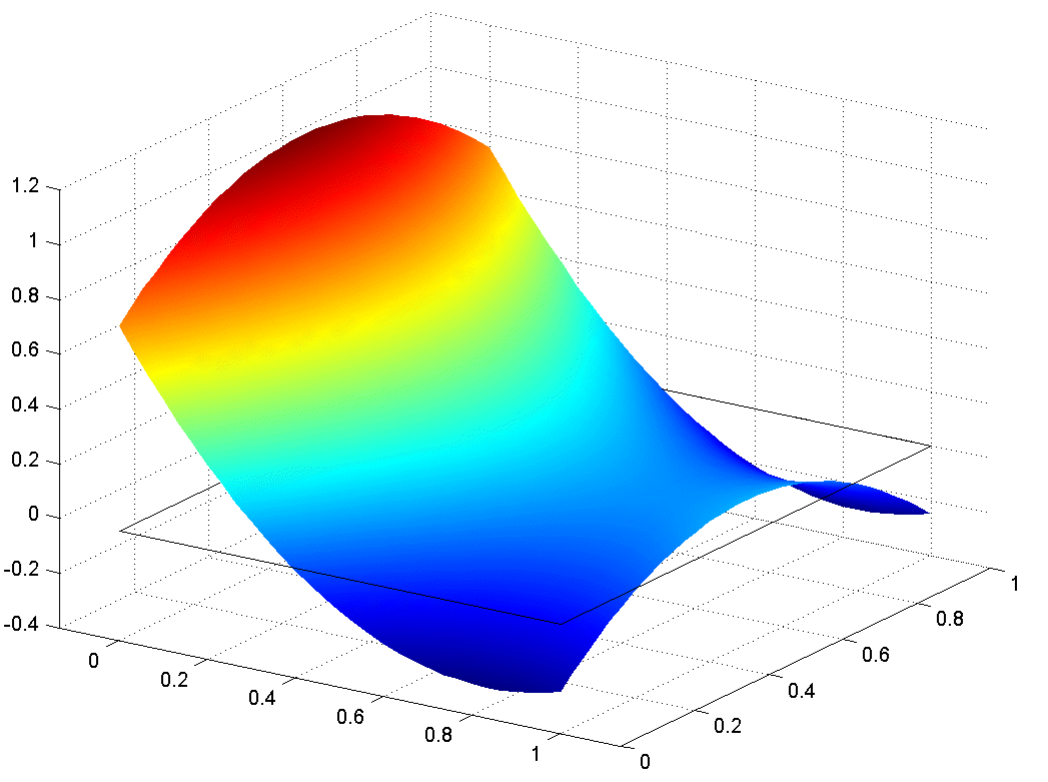}
  \includegraphics[width=0.3\textwidth]{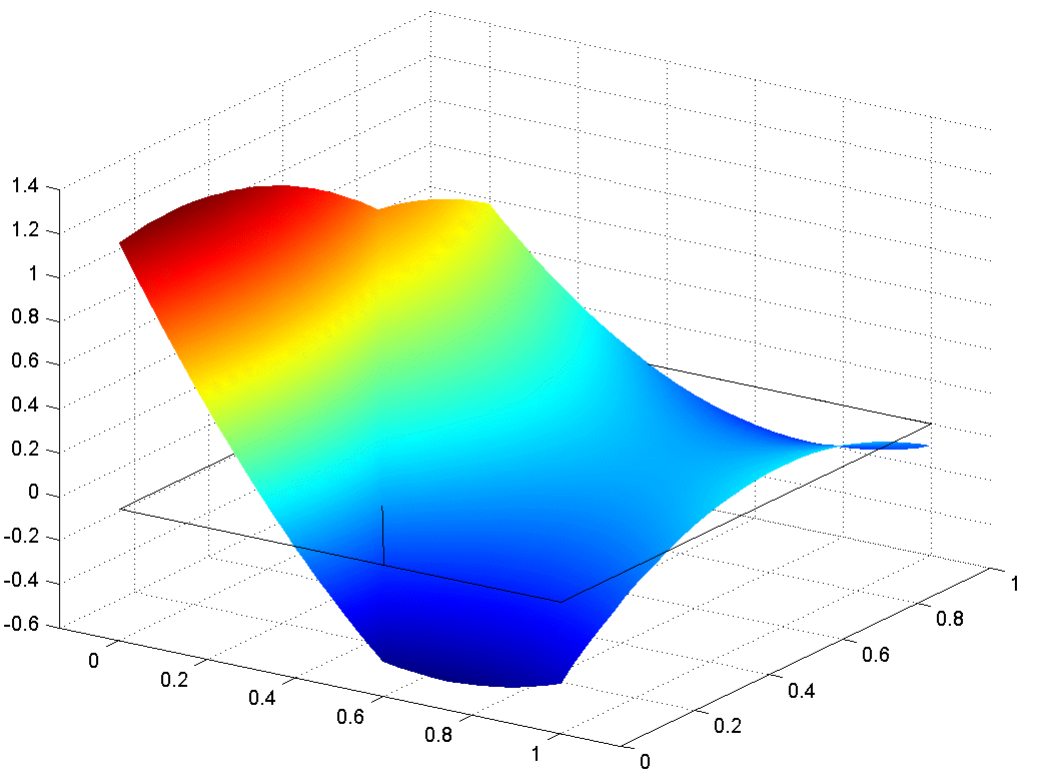}
  \includegraphics[width=0.3\textwidth]{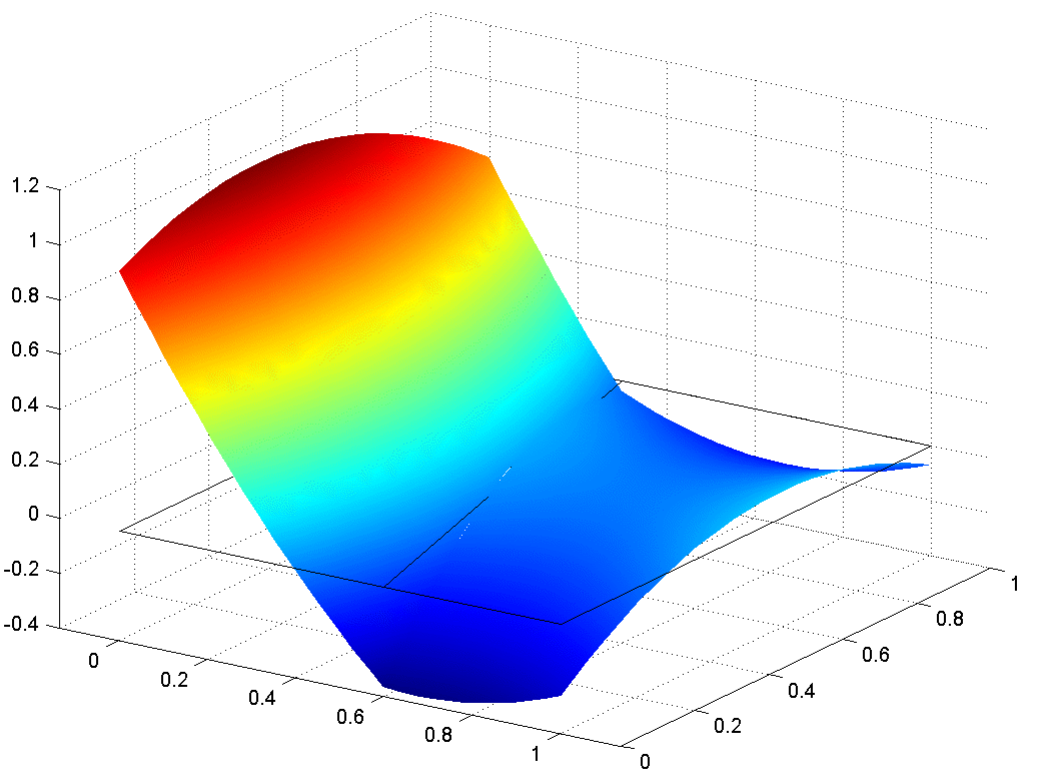}
  \caption{Nonconforming FE/IFE local bases with integral-value degrees of freedom}
  \label{fig: nonconforming FE IFE local basis integral}
\end{figure}



Denote by $ S_h^{i}(T) = \Span\{\phi_{j,T}:j = 1,2,3,4\}$ the local rotated-$Q_1$ IFE space on an interface element $T$ . 
The global IFE space is defined as follows
\begin{eqnarray}
  S_h(\Omega) &=& \Big\{v\in L^2(\O):~ v|_T\in S_h^n(T)\text{ if }T\in \mathcal{T}_h^n, v|_T\in S_h^{i}(T)\text{ if }T\in \mathcal{T}_h^i; \nonumber\\
   & &\qquad\qquad\qquad
  \int_\gamma [v]_\gamma ds = 0 \text{ for all interior edges } \gamma \text{ of } \mathcal{T}_h\Big\}.\label{eq: global IFE space integral}
\end{eqnarray}


\subsection{Basic Properties of IFE Space}
We summarize some basic and useful properties for the IFE space $S_h(\O)$ in this subsection. The results can be verified via straightforward calculation, and we also refer to \cite{2013ZhangTHESIS} for some of their proofs. 

\begin{lemma}\label{lemma: unisolvent elliptic}
{\bf (Unisolvency)} On each interface element $T \in \mathcal{T}_h^i$, an IFE function $\phi_T\in S_h^i(T)$  is uniquely determined by its integral values \eqref{eq: nonconforming IFE basis partial nodal} and interface jump conditions \eqref{eq: nonconforming IFE basis DE continuity} and \eqref{eq: nonconforming IFE basis flux continuity}.
\end{lemma}

\begin{lemma}\label{lemma: nonconforming IFE midpoint basis continuity}
{\bf (Continuity)} On each interface element $T \in \mathcal{T}_h^i$, the local IFE space $S_h^i(T)\subset C^0(T)$.
\end{lemma}

\begin{lemma}
\label{lemma: partition of unity}
{\bf (Partition of Unity)} On each interface element $T \in \mathcal{T}_h^i$, the IFE basis functions $\phi_{i,T}$ satisfy the partition of unity property, \emph{i.e.},
\begin{equation}\label{eq: partition of unity for Midpoint IFE basis}
    \sum_{j=1}^4\phi_{j,T}(x,y) = 1,~~~~\forall (x,y)\in T.
\end{equation}
\end{lemma}

\begin{lemma}\label{theorem: nonconforming IFE midpoint basis consistency}
{\bf (Consistency)} On each interface element $T \in \mathcal{T}_h^i$, the IFE basis functions are consistent to standard finite element basis functions in the following sense:
\begin{enumerate}
  \item If there is no jump in the coefficient, {\it i.e.}, $\beta^+ =
    \beta^-$, then the IFE basis functions  $\phi_{j,T}$ become the standard FE basis functions $\psi_{j,T}$.
  \item If $\min\{|T^+|,|T^-|\}$ shrinks to zero, then the IFE
basis functions $\phi_{j,T}$ become the standard FE basis
functions $\psi_{j,T}$. Here, $|T^s|$ denotes the area of $T^s$, $s=+,-$.
\end{enumerate}
\end{lemma}

\begin{lemma}\label{lemma: nonconforming IFE midpoint basis actual flux}
{\bf (Flux continuity on $\G$)}
On each interface element $T \in \mathcal{T}_h^i$, every IFE function $\phi_T\in S^i_h(T)$ satisfies the flux jump condition weakly as follows
\begin{equation*}
    \int_{\G\cap T} \djump{\bn\cdot\beta\nabla\phi_T}{\G\cap T}\mbox{d}s = 0,
\end{equation*}
where ${\bn}$ is the unit normal to $\G$.
\end{lemma}

\begin{lemma}\label{theorem: nonconforming IFE basis bound}
{\bf (Boundedness)} There exists a constant $C$, independent of
interface location, such that for $j=1,2,3,4$, and $k = 0,1,2$,
\begin{equation}\label{eq: IFE basis bound}
    \|\phi_{j,T}\|_{k,\infty,T} \leq Ch^{-k} ~~~\forall~ T \in \mathcal{T}_h^i.
\end{equation}
\end{lemma}
\begin{theorem}\label{theorem: trace inequality}
{\bf(Trace Inequality)} There exists a constant $C>0,$ depending
only on the diffusion coefficient $\beta,$ such that
\begin{equation}\label{eq: trace IFE}
    \|\bn\cdot\beta\nabla v\|_{0,\gamma} \leq Ch_T^{-\frac12}\|\nabla v\|_{0,T} ~~\forall v\in S_h^i(T)
\end{equation}
where $\gamma$ is an edge of $T$, and $\bnu$ is the unit outward normal to $T$.
\end{theorem}

\begin{theorem}\label{theorem: inverse inequality}
{\bf (Inverse Inequality)} There exists a constant $C$, depending only on the diffusion coefficient $\beta$, such that
\begin{equation}\label{eq: inverse IFE}
    |v|_{k,\infty,T} \leq Ch_T^{-1}|v|_{k,T},~~
    |v|_{k,T} \leq Ch_T^{l-k}|v|_{l,T}, ~~\forall v\in S_h^i(T),~~0\leq l\leq k\leq 2.
\end{equation}
\end{theorem}

\section{The Interpolation Operator and Approximation Capability}
In this section, we discuss the approximation capability for the nonconforming IFE
space $S_h(\Omega)$.
On each non-interface element $T\in\mathcal{T}_h^n$, the local
interpolation is defined  canonically by $\cI_{T}: C(\overline{T}) \to S_h^n(T),$
such that,
\begin{equation}\label{eq: local interpolation noninterface element}
    \cI_{T}u = \sum_{i=1}^4 \left(\frac{1}{|\cbe_i|}\int_{\cbe_i}u \ds\right)\psi_{i,T},
\end{equation}
where $\gamma_j$  denote the edges of $T$. The standard scaling argument leads to the
following error estimates \cite[Lemma 1]{1992RanacherTurek}:
\begin{equation}\label{eq: interpolation noninterface}
    \|\cI_{T}u-u\|_{0,T} + h|\cI_{T}u-u|_{1,T}\leq Ch^2|u|_{2,T}.
\end{equation}
On each interface element $T \in \mathcal{T}_h^i,$ the interpolation operator
$
\cI_{T}:C(\overline T) \to S_h^i(T)
$
is defined similarly as follows:
\begin{equation}\label{eq: interpolation IFE local}
    \cI_{T} u = \sum_{i=1}^4 \left(\frac{1}{|\cbe_i|}\int_{\cbe_i}u\ds\right)\phi_{i,T}~~\forall~ u \in C(\overline{T}).
\end{equation}
Finally, we define the global IFE interpolation
$\cI_h: C(\overline\O) \to
S_h(\O)$ piecewisely such that
\begin{equation*}
    (\cI_h u)|_{T} = \cI_{T} u\quad \forall\ T\in \mathcal{T}_h.
\end{equation*}

The error estimates for the interpolation operator on interface elements are reported in \cite{2016GuoLin, 2016GuoLinZhang, 2013ZhangTHESIS}. We only state the results in the following theorems.
\begin{theorem}
There exists a constant $C>0$, independent of interface location,
such that
\begin{equation}\label{eq: interpolation error integral}
    \|\cI_{T}u-u\|_{0,T} + h|\cI_{T}u-u|_{1,T}\leq
    Ch^2\|u\|_{\PHb^2(T)}
\forany u\in \PHb^2(T),
\end{equation}
on interface element $T\in\mathcal{T}_h^i$.
\end{theorem}
%
\begin{theorem}\label{th:I_h^I_bound}
There exists a constant $C>0$, such that the following interpolation error estimate holds:
\begin{equation}\label{eq: global interpolation error integral}
    \|\cI_{h}u-u\|_{0,\O} +
    h\left(\sum_{T\in\mathcal{T}_h}|\cI_{T}u-u|_{1,T}^2\right)^{\frac12}\leq
    Ch^2\|u\|_{\PHb^2(\O)}
\forany u\in \PHb^2(\O).
\end{equation}
\end{theorem}

\section{The IFE Galerkin Method and Error Estimates}
\label{sec: error analysis}

In this section, we consider a nonconforming IFE Galerkin method and
carry out its error estimation.

\subsection{The nonconforming IFE Galerkin method}
Given a mesh $\mathcal{T}_h$, we denote by $\cE$, $\cEz$ and $\cE^b$ the set of its edges, interior edges, and boundary edges, respectively. The
sets of interface edges and non-interface edges are denoted by $\cE^i$ and $\cE^n$, respectively. For the sake of simplicity, in the
following discussion, we assume that the interface curve $\G$ does not intersect the boundary $\partial \O$.
Consequently, $\cE^b\subset\cE^n$.

Define the bilinear and linear form
\begin{eqnarray*}
  a_h (u,v) = \sum_{T\in\mathcal{T}_h}\int_T\beta\nabla u \cdot \nabla v
  \dbx, \quad
  L(v) = \sum_{T\in\mathcal{T}_h}\int_T f v\dbx.
\end{eqnarray*}
The nonconforming IFE Galerkin method is to find $u_h\in S_h(\Omega)$ such that
\begin{equation}\label{eq: galerkin weak form}
    a_h (u_h,v_h) = L(v_h),~~~\forall~v_h\in \mathring{S}_h(\Omega),
\end{equation}
subject to the boundary conditions:
\begin{equation}\label{eq: bc}
    \int_\g u_h \ds = \int_\g g \ds,~~~~~\forall~\g\in \cE^b,
\end{equation}
and the test function space is defined as follows
\begin{eqnarray*}
  \mathring{S}_h(\O) &=& \{v\in S_h(\O): \int_{\cbe}v\ds = 0, \text{ if } \cbe\in \cE^b \}.
\end{eqnarray*}
In the following, we derive the error estimation of IFE solution $u_h$. 
\subsection{Projection operators}
\label{section: error Analysis}

For convenience in the analysis to follow, let
\[
\cbe_{jk} = \p T_j \cap \p T_k,\quad \cbe_j =
\p T_j\cap \p\O,\quad \forall\, T_j, T_k \in \Tau_h,
\]
and write
\[
v_j = v\at{T_j}\quad\forall\, T_j\in\Tau_h ;\quad
v_{jk} = v_j\at{\cbe_{jk}} \quad\forall \cbe_{jk}\in\cE^i.
\]
Set
\begin{eqnarray*}
\Lam^h&=&\left\{\lam \mid \lam = (\lam_{jk},\lam_{kj})\in
(\cP_0(\cbe_{jk}))^2,\lam_{jk}+\lam_{kj}=0~~\forall \cbe_{jk}\in
\mathring{\mathcal{E}}_h;\right. \\
&&\qquad\qquad\left. \lam = \lam_j\in \cP_0(\cbe_{j})~~\forall \cbe_{j}\in \mathcal{E}_h^b \right\}.
\end{eqnarray*}
Denote by $\bnu_j$ the unit outward normal to $T_j.$
We will use the following projection operators introduced in \cite{1999DouglasSantosSheenYe}:
$\Pi_0: \underset{\cbe\in\cE}{\Pi}L^2(\cbe)\to \underset{\cbe\in\cE}\Pi \cP_0(\cbe)$ and
$\Pi_\bnu:\PHb^2(\O)\to\Lam^h$ by
\begin{eqnarray}
  \left\la v - \Pi_0 v, 1\right\ra_\cbe &=& 0\quad \forall v\in
  L^2(\cbe)\quad \forall \cbe \in \cE,\label{eq: average value 1}\\
  \bigg\la \beta\pp{v_j}{\bnu_j}-\Pi_\bnu v,1\bigg\ra_{\cbe_{jk}} &=& 0 \quad\forall v\in\PHb^2(\O) \quad\forall \cbe_{jk} \in \cE,\label{eq: average value 2}
\end{eqnarray}
%
so that $\Pi_0^\cbe(v) := \Pi_0(v\at{\cbe}) = \frac1{|\cbe|} \int_\cbe v \ds$ is
the average of $v$ over $\cbe$ and
\[
\left(\Pi_\bnu v\at{\cbe_{jk}}, \Pi_\bnu v\at{\cbe_{kj}}\right) = \left( \Pi_0^{\gamma_{jk}} \beta\pp{v_j}{\bnu_j},
 \Pi_0^{\gamma_{kj}} \beta\pp{v_k}{\bnu_k}\right) \in \mathbb{R}^2
\quad\forall \cbe_{jk}\in \cE^i.
\]

\begin{lemma}\label{lem:fractional}
Let $\gamma = (0, h)$ with $\gamma^- = (0, \alpha)$ and $\gamma^+ = (\alpha, h)$. Assume
$u \in L^2(\gamma)$ and $u|_{\gamma^s}\in H^{\frac12}(\gamma^s),s = -, +.$
Then $u \in H^{\frac12-\eps}(\gamma)$ for every $\epsilon \in (0, \frac14)$, and
there exists a constant $C$ such that
\begin{eqnarray*}
\|u\|_{\frac12 - \epsilon,\gamma} \leq \frac{C}{\sqrt{\epsilon}}
\left(\|u\|_{\frac12,\gamma^-} + \|u\|_{\frac12, \gamma^+}\right).
\end{eqnarray*}
\end{lemma}
\begin{proof}
For every $\epsilon \in (0, \frac14)$, let $\sigma = \frac12 - \epsilon$. Then $\sigma \in (\frac14, \frac12)$. For $q \geq 1$ and $y \in \gamma^-$,
\begin{eqnarray*}
\int_{\gamma^+}\frac{1}{|x-y|^{(1+2\sigma)q}} dx &=&
\frac{(h-y)^{1-(1+2\sigma)q} - (\alpha-y)^{1-(1+2\sigma)q}}{1-(1+2\sigma)q}.
\end{eqnarray*}
Since $\sigma \in (\frac14, \frac12)$, we specifically choose
\begin{eqnarray*}
q = \frac{1}{2}\left(1 + \frac{2}{1+2\sigma}\right)~~~\Longrightarrow ~~~q = \frac{2-\epsilon}{2(1-\epsilon)}.
\end{eqnarray*}
Then $1 \leq q < \frac{2}{1+2\sigma}$, $1-(1+2\sigma)q \not = -1$, and $1+2\sigma \leq (1+2\sigma)q < 2$.
Hence,
\begin{eqnarray*}
&&I(\gamma^-,\gamma^+,\sigma,q):=\int_{\gamma^-}\int_{\gamma^+}\frac{1}{|x-y|^{(1+2\sigma)q}}dx dy \\
&=&\left(\frac{1}{1-(1+2\sigma)q}\right) \left(\frac{1}{2-(1+2\sigma)q}\right)\left(h^{2-(1+2\sigma)q} - (h-\alpha)^{2-(1+2\sigma)q}
- \alpha^{2-(1+2\sigma)q}\right) \\
&\leq& C \left|\frac{1}{1-(1+2\sigma)q}\right| \left|\frac{1}{2-(1+2\sigma)q}\right| = C \left|\frac{1}{1-\epsilon}\right| \left|\frac{1}{\epsilon}\right| \leq \frac{C}{\epsilon}.
\end{eqnarray*}
Therefore, using $p$ such that $\frac{1}{p} + \frac{1}{q} = 1$, and the above estimate, we have
\begin{eqnarray*}
&&\int_{\gamma^-}\int_{\gamma^+} \frac{\left| u(x) -
  u(y)\right|^2}{|x-y|^{1+2\sigma} } dxdy \\
&\leq& \int_{\gamma^-}\left(\left[\int_{\gamma^+}
\left| u(x) -  u(y)\right|^{2p}dx\right]^{\frac1{2p}}\right)^2
\left[\int_{\gamma^+}
\frac1{|x-y|^{(1+2\sigma)q}}dx\right]^{\frac1{q}}dy \\
&\le& 2\int_{\gamma^-}\|u\|_{0,2p,\gamma^+}^2 \left[\int_{\gamma^+}\frac1{|x-y|^{(1+2\sigma)q}}dx\right]^{\frac1{q}}dy
+ 2 \int_{\gamma^-}|u(y)|^2\left[\int_{\gamma^+}\frac1{|x-y|^{(1+2\sigma)q}}\dx\right]^{\frac1{q}}dy.\\
&\leq&2\|u\|_{0,2p,\gamma^+}^2 \left(\int_{\gamma^-} 1^p dy \right)^{\frac1p} I(\gamma^-,\gamma^+,\sigma,q)^{\frac1q} + 2\left(\int_{\gamma^-}|u(y)|^{2p} dy\right)^{\frac1p}
I(\gamma^-,\gamma^+,\sigma,q)^{\frac1q}\\
&\le& C\|u\|_{0,2p,\gamma^+}^2 \left(\frac{1}{\epsilon}\right)^{\frac1q} + C\|u\|_{0,2p,\gamma^-}^2 \left(\frac{1}{\epsilon}\right)^{\frac1q}\\
&\le& C\left(\frac{1}{\epsilon}\right)^{\frac{2(1-\epsilon)}{2-\epsilon}} \left(\|u\|_{\frac12,\gamma^-}^2 + \|u\|_{\frac12, \gamma^+}^2\right).
\end{eqnarray*}
In the last step, we used the Sobolev embedding theorem for one dimension:
\begin{equation*}
  W^{\frac12,2}(\gamma^s)\hookrightarrow W^{0,p}(\overline{\gamma^s}),~~~s = -,+,~~~p\in [1,\infty).
\end{equation*}
By definition of the fractional Sobolev norm, we have
\begin{eqnarray*}
&&\|u\|_{\sigma,\gamma}^2 = \|u\|_{0,\gamma}^2+ \int_\gamma\int_\gamma \frac{\left| u(x) - u(y)\right|^2}{|x-y|^{1+2\sigma} } dxdy\\
&=& \|u\|_{0,\gamma}^2+   \int_{\gamma^-}\int_{\gamma^-} \frac{\left| u(x) -
  u(y)\right|^2}{|x-y|^{1+2\sigma} } dxdy
+2\int_{\gamma^-}\int_{\gamma^+} \frac{\left| u(x) -
  u(y)\right|^2}{|x-y|^{1+2\sigma} } dxdy  \\
&& ~~~~ +\int_{\gamma^+}\int_{\gamma^+} \frac{\left| u(x) -
  u(y)\right|^2}{|x-y|^{1+2\sigma} } dxdy \\
&\leq& \|u\|_{\frac12,\gamma^-}^2 + \|u\|_{\frac12, \gamma^+}^2 + C\left(\frac{1}{\epsilon}\right)^{\frac{2(1-\epsilon)}{2-\epsilon}} \left(\|u\|_{\frac12,\gamma^-}^2 + \|u\|_{\frac12, \gamma^+}^2\right),
\end{eqnarray*}
which leads to
\begin{eqnarray*}
\|u\|_{\sigma,\gamma} \leq \frac{C}{\epsilon^{(1-\epsilon)/(2-\epsilon)}}\left(\|u\|_{\frac12,\gamma^-} + \|u\|_{\frac12, \gamma^+}\right) \leq \frac{C}{\sqrt{\epsilon}}
\left(\|u\|_{\frac12,\gamma^-} + \|u\|_{\frac12, \gamma^+}\right)
\end{eqnarray*}
because for small $\epsilon$, we have
\begin{equation*}
\frac{1-\epsilon}{2-\epsilon} = \frac{1}{2} - \frac{\epsilon}{4} - \frac{\epsilon^2}{8} - \frac{\epsilon^3}{16} - \cdots \leq \frac{1}{2}.
\end{equation*}
\end{proof}

\begin{theorem} \label{th:flux_edge_approx_bnd}
Let $T \in \mathcal{T}_h$ and let $\gamma$ be an edge of $T$. Then there exists a constant $C>0$ such that the following hold on a mesh
$\mathcal{T}_h$ with a sufficiently small mesh size:
\begin{enumerate}
\item
If $T \in \mathcal{T}_h^n$ and $v \in H^2(T)+ S_h^n(T)$, then
\begin{eqnarray*}
\|\beta\pp{v}{\bnu}-\Pi_\bnu v\|_{0, \gamma} \leq Ch^{\frac12}\|v\|_{H^2(T)}.
\end{eqnarray*}

\item
If $\gamma \in \cE^n$ but $T \in \mathcal{T}_h^i$, and $v \in \PHb^2(T) + S_h^i(T)$, then
\begin{eqnarray*}
\|\beta\pp{v}{\bnu}-\Pi_\bnu v\|_{0, \gamma} \leq Ch^{\frac12}\left(\|v\|_{H^2(\tT^-)} + \|v\|_{H^2(\tT^+)}\right).
\end{eqnarray*}

\item
If $\gamma \in \cE^i$ and $v \in \PHb^2(T) + S_h^i(T)$, then
\begin{eqnarray*}
\|\beta\pp{v}{\bnu}-\Pi_\bnu v\|_{0, \gamma} \leq C h^{\frac12} |\log h|^{\frac12}
 \left(\|v\|_{H^2(\tT^-)} + \|v\|_{H^2(\tT^+)}\right).
\end{eqnarray*}
\end{enumerate}
Here, for $T \in \mathcal{T}_h^i$, designate
\begin{eqnarray*}
\tT^s = \begin{cases}
T \cap \Omega^s &\text{~for~} v\in \PHb^2(T), \\
T^s & \text{~for~} v\in S_h(T),
\end{cases}
\qquad\text{ for }s = \pm.
\end{eqnarray*}
\end{theorem}
\begin{proof}
Let $\gamma \in \cE$. In the first two cases we assume $\gamma \in \cE^n$, but for the third case we assume
$\gamma \in \cE^i$. Then by the standard trace theorem or the lemma above, we have
$\beta\pp{v}{\bnu} \in H^{\frac12}(\gamma)$ or $\beta\pp{v}{\bnu} \in
H^{\frac12-\eps}(\gamma)$ for any $\eps \in (0, \frac14)$.

Since $\Pi_\bnu v$ is the $L^2$ projection of $\beta\pp{v}{\bnu}$ to the space of constant polynomials, applying the
error estimate for polynomial projection and the standard error estimate on interpolation of Sobolev spaces (see \cite[Theorem 1.4, p.6]{1986GiraultRaviart}),
we have
\begin{eqnarray} \label{eq:trace_interp_1}
\left\|\beta\pp{v}{\bnu}-\Pi_\bnu v\right\|_{0, \gamma} &\leq&
\begin{cases}
Ch^{\frac12}\left\|\beta\pp{v}{\bnu}\right\|_{\frac12, \gamma} & \text{if~} \gamma \in \cE^n, \\ \\
Ch^{\frac12-\eps}\left\|\beta\pp{v}{\bnu}\right\|_{\frac12-\eps, \gamma} & \text{if~} \gamma \in \cE^i.
\end{cases}
\end{eqnarray}
For the first two cases, by the definition of $\|\cdot\|_{\frac12, \gamma}$, we have
\begin{eqnarray*}
\left\|\beta\pp{v}{\bnu}\right\|_{\frac12, \gamma} &\leq& \begin{cases}
\left\| \beta\pp{v}{\bnu} \right\|_{1, T} & \text{if~} T \in \mathcal{T}_h^n, \\ \\
\left\|\beta\pp{v}{\bnu}\right\|_{1, \tT^s} \leq \left\|\beta\pp{v}{\bnu}\right\|_{1, \tT^-} + \left\|\beta\pp{v}{\bnu}\right\|_{1, \tT^+} & \text{if~} T \in \mathcal{T}_h^i,
\end{cases}
\end{eqnarray*}
which means
\begin{eqnarray} \label{eq:trace_interp_2}
\left\|\beta\pp{v}{\bnu}\right\|_{\frac12, \gamma} &\leq& \begin{cases}
\max\{\beta^+,\beta^-\}\left\|v\right\|_{2, T} & \text{if~} T \in \mathcal{T}_h^n, \\ \\
\max\{\beta^+,\beta^-\}\left\|v\right\|_{\PHb^2(T)} & \text{if~} T \in \mathcal{T}_h^i.
\end{cases}
\end{eqnarray}
For the third case, applying Lemma \ref{lem:fractional}, we have
\begin{eqnarray}\label{eq:trace_interp_3}
\left\|\beta\pp{v}{\bnu}\right\|_{\frac12-\eps, \gamma} &\leq& \frac{C}{\sqrt{\eps}}
\left(\|\beta\pp{v}{\bnu}\|_{\frac12,\gamma^-} + \|\beta\pp{v}{\bnu}\|_{\frac12, \gamma^+}\right)  \nonumber \\
&\leq& \frac{C}{\sqrt{\eps}}
\left(\|v\|_{2,\tT^-} + \|v\|_{2, \tT^+}\right).
\end{eqnarray}
Finally, all the estimates in this theorem follow by applying \eqref{eq:trace_interp_2} and \eqref{eq:trace_interp_3}
to \eqref{eq:trace_interp_1}, and by taking the minimum of
$\frac{1}{h^{\eps}\sqrt{\eps}}$ over $0<\eps< 1/4.$ Indeed,
at $\eps=\frac1{2\log\frac1{h}}$, for $0 < h < \frac{1}{e^2}$, the minimum value is
\begin{equation*}
  \frac{1}{h^{\eps}\sqrt{\eps}} = h^{\frac1{2\log h}}\sqrt{{2\log\frac1{h}}} = \sqrt{2e}|\log h|^\frac12.
\end{equation*}
\end{proof}

\subsection{The Energy-Norm Error Estimate}\label{sec3}
Define the (broken) energy norm
\[
\tbar u\tbar=\sqrt{a_h(u,u)}.
\]
As needed, we quote the following second Strang lemma for the IFE solution:
\begin{lemma}\label{lem3.1} Let
$u\in \tHi^1(\O)$ and $u_h\in S_h(\O)$ be the solutions of
\eqref{eq:weakform} and \eqref{eq: galerkin weak form}, respectively.
Then,
\begin{eqnarray}\label{eq:strang}
\tbar u-u_h\tbar\leq C\left\{\inf_{v_h\in S_h(\O)}\tbar u-v_h\tbar+\sup_{w_h\in S_h(\O)}
\frac{|a_h(u,w_h)-L(w_h)|}{\tbar w_h\tbar}\right\}.
\end{eqnarray}
\end{lemma}

We are now ready to state and derive an error estimate in the energy norm.

\begin{theorem}\label{thmh1} Let
$u\in \PHb^2(\O)$ and $u_h\in S_h(\O)$ be the solutions of
\eqref{eq:weakform} and \eqref{eq: galerkin weak form}, respectively.
Then, there exists a constant $C$ such that
\begin{equation}\label{eq: energy error estimate}
\tbar u-u_h\tbar\leq C h\left[ \|u\|_{\PHb^2(\O)} +
 |\log h|^\frac12\sum_{T \in \mathcal{T}_h^i}\|u\|_{\PHb^2(T)}\right].
\end{equation}
If, in addition, $u\in \widetilde W^{2,q}_\beta(\Omega)$ for some
$q>2,$ then there exists $h_0>0$ such that, for all $0<h<h_0,$
\begin{equation}\label{eq: energy error estimate 2}
\tbar u-u_h\tbar\leq C h\Big( \|u\|_{\PHb^2(\O)} +
 \sum_{T \in \mathcal{T}_h^i}\|u\|_{\tilde W^{2,q}_\beta(T)}\Big).
\end{equation}
\end{theorem}

\begin{proof}
We need to estimate those terms bounding $\tbar u-u_h\tbar$ in \eqref{eq:strang} of the Strang lemma above. By
the interpolation estimate \eqref{eq: global interpolation error integral}, we can estimate the first term
on the right hand side of \eqref{eq:strang} as follows:
\begin{eqnarray}\label{eq:interpol}
\inf_{v_h\in S_h(\O)}\tbar u-v_h\tbar\leq Ch\|u\|_{\PHb^2(\O)}.
\end{eqnarray}
Next, let $w_h\in S_h(\O)$ be arbitrary. Then,
since $u\in H^1(\O)$, it follows that
\begin{eqnarray*}
a_h(u,w_h)&=&\sum_j(\beta\grad u, \grad w_{h})_{T_j}
=-\sum_j (\div\beta\grad u, w_h)_{T_j} + \sum_j\bigg\la
\beta\pp{u_j}{\bnu_j},w_h\bigg\ra_{\p T_j}\\
&=&
 (-\div\beta\grad u, w_h) + \sum_j\bigg\la \beta\pp{u_j}{\bnu_j},w_h\bigg\ra_{\p T_j}.
\end{eqnarray*}
Hence, by choosing
$m_j\in\cP_0(T_j)$ to be the the average of $w_h$ over $T_j,$ one sees that
\begin{eqnarray}
 a_h(u,w_h)-L(w_h)&=&\sum_j\bigg\la \beta\pp{u_j}{\bnu_j},w_h\bigg\ra_{\p
  T_j}
=\sum_j\bigg\la \beta\pp{u_j}{\bnu_j}-\Pi_\bnu u_j, w_h\bigg\ra_{\p T_j} \nonumber\\
&=& \sum_j\bigg\la \beta\pp{u_j}{\bnu_j}-\Pi_\bnu u_j, w_{h}-m_j\bigg\ra_{\p T_j}. \label{eq:a_h-L}
\end{eqnarray}
Hence, by Theorem \ref{th:flux_edge_approx_bnd}, the trace inequality on $T_j$, and the approximation capability of $m_j$, we have
\begin{eqnarray}
\qquad&&\left|a_h(u,w_h)-L(w_h)\right|\nonumber\\
&\le&  \Big(\sum_{j} \left\|\beta\pp{u_j}{\bnu_j}-\Pi_\bnu u_j\right\|_{0, \partial T_j}^2\Big)^{\frac12}\Big(\sum_j\|w_h-m_j\|_{0,\partial T_j}^2\Big)^{\frac12} \nonumber \\
&\le& \Big(Ch^{\frac12}\sum_{T \in \mathcal{T}_h^n} \|u\|_{2,T} + {C}h^{\frac12}|\log h|^\frac{1}{2} \sum_{T \in \mathcal{T}_h^i} \|u\|_{\PHb^2(T)}\Big)
Ch^{\frac12}\Big(\sum_j\|\nab w_{h}\|^2_{0,T_j}\Big)^{\frac12} \nonumber \\
&\leq& Ch \Big(\sum_{T \in \mathcal{T}_h^n} \|u\|_{2,T} + |\log h|^\frac{1}{2}\sum_{T \in \mathcal{T}_h^i} \|u\|_{\PHb^2(T)}\Big)\tbar w_h\tbar.
\label{eq13}
\end{eqnarray}
Then, applying \eqref{eq:interpol} and \eqref{eq13} to
\eqref{eq:strang} leads to \eqref{eq: energy error estimate}.

Assume that $u\in \widetilde W^{2,q}_\beta(\Omega)$ for some
$q>2.$ Then choose $p$ such that $\frac{1}{p} + \frac{2}{q} =1,$
so that, for $T\in\mathcal{T}_h^i$
\begin{eqnarray*}
  \|u\|_{\PHb^2(T)}
  &\leq& \Big(\sum_{s=\pm}\int_{T^s} \sum_{|\alpha|\leq 2} |D^\alpha u|^2dx\Big)^{\frac12} \leq
  \left(\int_{T} 1^pdx\right)^\frac{1}{2p}\Big(\sum_{s=\pm}\int_{T^s}\Big(\sum_{|\alpha|\leq 2} |D^\alpha u|^2\Big)^{\frac{q}{2}}dx\Big)^\frac{1}{q} \\
   &\leq& C|T|^\frac{1}{2p}\Big(\sum_{s=\pm}\int_{T^s}\sum_{|\alpha|\leq 2} |D^\alpha u|^qdx\Big)^\frac{1}{q}\leq Ch^{\frac{1}{p}}\|u\|_{\tilde W^{2,q}_\beta(T)}.
\end{eqnarray*}
Hence, the second term in \eqref{eq: energy error estimate} can be bounded by
\begin{equation*}
  |\log h|^\frac12\sum_{T \in \mathcal{T}_h^i}\|u\|_{\PHb^2(T)} \leq
  C\sum_{T \in \mathcal{T}_h^i}|\log
  h|^\frac12h^{\frac{1}{p}}\|u\|_{\tilde W^{2,q}_\beta(T)}.
\end{equation*}
Since $\lim_{h\to 0}|\log h|^\frac12h^{\frac{1}{p}}=0,$ there exists
$h_0>0$ such that
the estimate \eqref{eq: energy error estimate 2} is valid for
$0<h<h_0$. This completes the proof.
\end{proof}

\begin{remark}
Indeed, \eqref{eq: energy error estimate} implies that the IFE solution converge faster than $\mathcal{O}(h|\log h|^\frac12)$, since its multiplication factor, $\sum_{T \in \mathcal{T}_h^i}\|u\|_{\PHb^2(T)}$, goes to zero as $h\to 0$.
\end{remark}

\subsection{Duality and the $L^2$-Error Estimate}\label{sec4}
Let
\[
\eta_h = \cI_{h} u-u_h\in \mathring{S}_h(\O),
\]
and let $\psi\in \PHb^2(\O)$ be the solution of the dual problem:
\begin{subeqnarray} \label{eq:dual}
-\nab\cdot(\beta\nab\psi)&=& \eta_h \quad \text{in }\O, \\
\psi &=& 0 \quad\  \text{on }\p\O.
\end{subeqnarray}
Assume that the interface problem \eqref{eq:weakform} is $\PHb^2(\O)$-regular
so that the elliptic regularity estimate holds:
\begin{eqnarray}
\|\psi\|_{\PHb^2(\O)}\leq C\|\eta_h\|_0. \label{eq:regularity_assumption}
\end{eqnarray}
We start from recalling the following standard estimates for the IFE interpolation $\cI_h \psi$: there exists a constant $C$ such that
\begin{eqnarray}\label{eq:IFE_interp_bounds}
\begin{cases}
\|\cI_h \psi\|_{2,T} \leq C \left\|\psi\right\|_{2,T} &\forall ~T \in \mathcal{T}_h^n, \\
\|\cI_h \psi\|_{2, T_j^-} + \|\cI_h \psi\|_{2, T_j^+} \leq C\left\|\psi \right\|_{{\PHb^2(T)}} &\forall~ T\in \mathcal{T}_h^i.
\end{cases}
\end{eqnarray}
Since $\eta_h\in \mathring{S}_h(\O),$ it follows that
\begin{eqnarray*}
\|\eta_h\|_0^2&=&(-\div\beta\grad\psi,\eta_h)=a_h(\psi,\eta_h)-\sum_j\bigg\la
\beta\pp{\psi_j}{\bnu_j},\eta_{h_j}\bigg\ra_{\p T_j}\\
&=&a_h(\psi,\eta_h)-\sum_j\bigg\la
\beta\pp{\psi_j}{\bnu_j}-\Pi_\bnu \psi_j,\eta_{h_j}-q_j\bigg\ra_{\p
  T_j}
\quad\text{for all }q_j\in\cP_0(T_j).
\end{eqnarray*}
 Next, for all $v_h\in\mathring{S}_h(\O)$, similarly to \eqref{eq:a_h-L}, we have
\begin{eqnarray*}
a_h(\eta_h,v_h)&=&a_h(u,v_h)-a_h(u_h,v_h)-a_h(u-\cI_h u,v_h) \\
&=&\sum_j\bigg\la
\beta\pp{u_j}{\bnu_j}-\Pi_\bnu u_j,v_{h_j}\bigg\ra_{\p T_j}-a_h(u-\cI_h u,v_h).
\end{eqnarray*}
Using the property $\jump{\psi}{\cbe_{jk}}=0$ and recalling the
definition of $\Pi_\bnu$, we see that
\[
\left\la \beta\pp{u_j}{\bnu_j}-\Pi_\bnu u_j,\psi_j\right\ra_{\cbe_{jk}} +
\left\la \beta\pp{u_k}{\bnu_k}-\Pi_\bnu u_k,\psi_k\right\ra_{\cbe_{kj}} = 0.
\]
In addition, note that for $v_h\in S_h(\O)$, $-\div(\beta\grad v_h) = 0$ on every $T\in\mathcal{T}_h$; hence,
\begin{eqnarray*}
a_h(u-\cI_h u,v_h) &=& \sum_j(u-\cI_h u,-\div(\beta\grad v_h))_{T_j}+\sum_j\bigg\la u-\cI_h u,\beta\pp{v_h}{\bnu_j}\bigg\ra_{\p T_j} \\
&=& \sum_j\bigg\la u-\cI_h u,\beta\pp{v_h}{\bnu_j} - \Pi_{\bnu_j} v_h\bigg\ra_{\p T_j}.
\end{eqnarray*}
Therefore
\begin{eqnarray}
\|\eta_h\|_0^2&=& a_h(\psi,\eta_h)-\sum_j\bigg\la
\beta\pp{\psi_j}{\bnu_j}-\Pi_\bnu \psi_j,\eta_{h_j}-q_j\bigg\ra_{\p T_j} \nonumber \\
&=& a_h(\eta_h,\psi-v_h)-a_h(u-\cI_h u,v_h)
\nonumber\\
&-&\sum_j\bigg\la
\beta\pp{\psi_j}{\bnu_j}-\Pi_\bnu \psi_j,\eta_{h_j}-q_j\bigg\ra_{\p T_j}
+\sum_j\bigg\la \beta\pp{u_j}{\bnu_j}-\Pi_\bnu u_j,v_{h_j}-\psi_j\bigg\ra_{\p T_j} \nonumber \\
&=& a_h(\eta_h,\psi-v_h)-\sum_j\bigg\la u-\cI_h u,\beta\pp{v_h}{\bnu_j} - \Pi_{\bnu_j} v_h\bigg\ra_{\p T_j} \nonumber\\
&-&\sum_j\bigg\la
\beta\pp{\psi_j}{\bnu_j}-\Pi_\bnu \psi_j,\eta_{h_j}-q_j\bigg\ra_{\p T_j}
+\sum_j\bigg\la \beta\pp{u_j}{\bnu_j}-\Pi_\bnu u_j,v_{h_j}-\psi_j\bigg\ra_{\p T_j}.
\label{eq5.6}
\end{eqnarray}

With these preparations, we are ready to derive the error estimate in the $L^2$--norm for the IFE solution.

\begin{theorem}\label{thml2} Assume the interface problem \eqref{eq:weakform}
is $\PHb^2(\O)$-regular. Then, there exists a constant $C$ such that
the $L^2$-norm error of the IFE solution satisfies the following
estimate:
\begin{equation}\label{eq: L2 error estimate}
\|u-u_h\|_0 \leq Ch^2\left[|\log h|^{\frac{1}{2}}\|u\|_{\PHb^2(\O)} + |\log h|\sum_{T \in \mathcal{T}_h^i}\|u\|_{\PHb^2(T)}\right].
\end{equation}
\end{theorem}
\begin{proof}
We proceed to estimate each term on the right hand side of \eqref{eq5.6}. First, choose $v_h = \cI_h \psi$. Then,
by \eqref{eq: global interpolation error integral} and \eqref{eq:regularity_assumption},
the first term on the right-hand side of \eqref{eq5.6} is bounded as follows:
\begin{eqnarray}\label{eq:etah,psi-vh}
|a_h(\eta_h,\psi-v_h)| = |a_h(\eta_h,\psi-\cI_h\psi)|\leq C h \tbar \eta_h\tbar\|\eta_h\|_0.
\end{eqnarray}
Again, choosing $q_j\in\cP_0(T_j)$ to be the the average of $\eta_h$ over $T_j$, by Theorem \ref{th:flux_edge_approx_bnd},
the trace inequality on
$T_j$, Theorem \ref{th:I_h^I_bound}, and \eqref{eq:regularity_assumption}, we can bound the last two terms on the right hand side of
\eqref{eq5.6} as follows:
\begin{eqnarray}
&&\bigg|\sum_j\bigg\la
\beta\pp{\psi_j}{\bnu_j}-\Pi_\bnu \psi_j,\eta_{h_j}-q_j\bigg\ra_{\p
  T_j}\bigg| + \bigg|\sum_j\bigg\la
\beta\pp{u_j}{\bnu_j}-\Pi_\bnu u_j,v_{h_j}-\psi_j\bigg\ra_{\p
  T_j}\bigg| \nonumber\\
&&\qquad\leq Ch \Big(\sum_{T \in \mathcal{T}_h^n} \|\psi\|_{2,T} + |\log h|^\frac12\sum_{T \in \mathcal{T}_h^i} \|\psi\|_{\PHb^2(T)}\Big)
\tbar \eta_h\tbar \nonumber \\
&&\qquad\qquad\qquad + Ch^2 \Big(\sum_{T \in \mathcal{T}_h^n} \|u\|_{2,T} + |\log h|^\frac12\sum_{T \in \mathcal{T}_h^i} \|u\|_{\PHb^2(T)}\Big)
\|\psi\|_{\PHb^2(\O)} \nonumber \\
&&\qquad\leq  Ch \Big(|\log h|^\frac12 \tbar \eta_h\tbar + h\sum_{T \in
    \mathcal{T}_h^n} \|u\|_{2,T} +
h|\log h|^\frac12\sum_{T \in \mathcal{T}_h^i} \|u\|_{\PHb^2(T)}\Big)
\|\eta_h\|_{0}.
\label{eq:betadpsidn}
\end{eqnarray}
For the second term in \eqref{eq5.6}, by Theorem \ref{th:I_h^I_bound}, Theorem \ref{th:flux_edge_approx_bnd}, \eqref{eq:IFE_interp_bounds} and
\eqref{eq:regularity_assumption},
\begin{eqnarray}\label{eq:2nd_term}
\qquad&&\Big|\sum_j\bigg\la u-\cI_h u,\beta\pp{v_{h_j}}{\bnu_j} - \Pi_{\bnu_j} v_h\bigg\ra_{\p T_j}\Big| \nonumber \\
&\leq& \Big(\sum_j \left|u-\cI_h u\right|_{0, \p T_j}^2\Big)^{\frac12} \Big(\sum_j \left|\beta\pp{v_{h_j}}{\bnu_j} - \Pi_{\bnu_j} v_h\right|_{0,\p T_j}^2\Big)^{\frac12} \nonumber \\
&\leq& Ch^{\frac32}\|u\|_{\PHb^2(\O)}\Big(h\sum_{T \in
    \mathcal{T}_h^n}\|v_h\|_{\PHb^2(T)}^2 + h|\log h|
\sum_{T \in \mathcal{T}_h^i}\big(\|v_h\|_{2,T^-}^2 + \|v_h\|_{2,T^+}^2\big) \Big)^{\frac12} \nonumber \\
&\leq&C h^2|\log h|^\frac12\|u\|_{\PHb^2(\O)}\|\eta_h\|_0.
\end{eqnarray}
Plugging the estimates \eqref{eq:etah,psi-vh}--\eqref{eq:2nd_term} in
\eqref{eq5.6}
gives
\begin{eqnarray*}
\|\eta_h\|_0
&\leq& Ch|\log h|^\frac12\tbar \eta_h\tbar + Ch^2|\log h|^\frac12\|u\|_{\PHb^2(\O)} \\
&\leq& Ch|\log h|^\frac12 \left(\tbar \cI_hu - u\tbar + \tbar u -
u_h\tbar\right)
 + C h^2 |\log h|^\frac12\|u\|_{\PHb^2(\O)}.
\end{eqnarray*}
Finally, applying Theorem \ref{th:I_h^I_bound} and Theorem \ref{thmh1}
to the above estimate, we arrive at the desired estimate
\eqref{eq: L2 error estimate}. This completes the proof. \end{proof}

\begin{remark}
The estimate given in \eqref{eq: L2 error estimate} suggests that the IFE solution converges in $L^2$-norm better than $\mathcal{O}(h^2|\log h|)$ which is optimal sans the usual $|\log h|$ factor.
\end{remark}

\begin{remark}
An optimal rate $O(h^2)$ without $|\log h|$ factor may be obtained with slightly better regularity $u\in \tilde W^{2,q}_\beta(\Omega)$, $q>2$, and the elliptic regularity assumption based on $L^q$-norm. In addition, the analysis requires the interpolation error estimates for IFE functions based on $L^q$-norm, which will be an interesting future work.
\end{remark}

\section{Numerical Examples}
\label{sec: numerical examples}
In this section, we present numerical examples to demonstrate
the features of this nonconforming rotated-$Q_1$ IFE method for elliptic interface
problems.

We test these IFE methods with the same example as given in
\cite{2008HeLinLin, 2013ZhangTHESIS}.
Let $\O = (-1,1)^2$, and the interface curve $\G$ is the
circle centered at the origin with radius $r_0 = \pi/6.28$,
which separates the domain into two sub-domains:
\begin{equation*}
    \O^- = \{(x,y)\in \O: x^2+y^2<r_0^2\},~~~\O^+ = \{(x,y)\in \O:~ x^2+y^2>r_0^2\}.
\end{equation*}
The boundary condition $g$ and the source function $f$ are chosen
such that the exact solution is as follows:
\begin{equation}\label{eq: exact solution}
    u(x,y) =
    \left\{
      \begin{array}{ll}
        \dfrac{r^a}{\beta^-} & \text{if } r<r_0, \\
        \dfrac{r^a}{\beta^+} + \left(\dfrac{1}{\beta^-} - \dfrac{1}{\beta^+}\right)r_0^a & \text{if } r>r_0,
      \end{array}
    \right.
\end{equation}
where $a = 5$, $r = \sqrt{x^2+y^2}$. We use a family of Cartesian
meshes $(\mathcal{T}_h)_{0<h<1},$ each of which consists of
$N\times N$ congruent squares of size $h = 2/N$.
Errors of an IFE approximation are given in $L^\infty$-, $L^2$-,
and semi $H^1$- norms.
Error in $L^\infty$-norm is calculated using the formula:
\begin{equation}\label{eq: discrete infinity norm error}
    \|u_{h}-u\|_{L^\infty} = \max_{T\in \mathcal{T}_h}\bigg(\max_{(x,y)\in{\hat T} \subset T} |u_{h}(x,y)-u(x,y)|\bigg),
\end{equation}
where $\hat T$ consists of the $49$ uniformly distributed points in $T$.
The $L^2$ and semi $H^1$ norms are computed using the 9-point Gaussian
quadratures.

Our first numerical experiment considers a moderate coefficient jump $(\beta^-,\beta^+) = (1,10)$.  Errors of numerical solutions are reported in Table
\ref{table: Nonconforming IFE integral Galerkin integral 1 10}.  Convergence
rates in semi $H^1$-norm and $L^2$-norm confirm our error analysis \eqref{eq: energy error estimate} and
\eqref{eq: L2 error estimate}. Data in these tables also suggest that the convergence rate in $L^\infty$-norm are approximately $O(h^2)$, which is also optimal from the point view of the degree of polynomials in constructing IFE spaces $S_h(\Omega)$.

\begin{table}[!hbt]
\begin{small}
\caption{Errors of Galerkin IFE solutions $u-u_{h}$ with $\beta^-=1$, $\beta^+=10$}
\begin{center}
\begin{tabular}{|c|cc|cc|cc|}
\hline
$N$ & $\|\cdot\|_{L^\infty}$ & rate & $\|\cdot\|_{L^2}$ & rate & $|\cdot|_{H^{1}}$ & rate \\
\hline
$10$      &$ 2.6183E{-2}$ & &$ 1.1395E{-2}$ & &$ 1.9585E{-1}$ &\\
$20$      &$ 7.3444E{-3}$ & 1.8339&$ 2.9860E{-3}$ & 1.9321&$ 9.9065E{-2}$ & 0.9833\\
$40$      &$ 1.9455E{-3}$ & 1.9165&$ 7.4374E{-4}$ & 2.0054&$ 4.9894E{-2}$ & 0.9895\\
$80$      &$ 5.0072E{-4}$ & 1.9580&$ 1.8547E{-4}$ & 2.0036&$ 2.5026E{-2}$ & 0.9955\\
$160$     &$ 1.2702E{-4}$ & 1.9789&$ 4.6313E{-5}$ & 2.0017&$ 1.2531E{-2}$ & 0.9979\\
$320$     &$ 3.1989E{-5}$ & 1.9894&$ 1.1671E{-5}$ & 1.9885&$ 6.2702E{-3}$ & 0.9990\\
$640$     &$ 8.0267E{-6}$ & 1.9947&$ 2.9122E{-6}$ & 2.0027&$ 3.1363E{-3}$ & 0.9995\\
$1280$    &$ 2.0101E{-6}$ & 1.9975&$ 7.2684E{-7}$ & 2.0024&$ 1.5684E{-3}$ & 0.9997\\
\hline
\end{tabular}
\end{center}
\label{table: Nonconforming IFE integral Galerkin integral 1 10}
\end{small}
\end{table}

The error surface $e_h = |u_h - u|$  are reported in Figure \ref{fig: comparison point-wise error}. For comparison, we also plot the error surface of Lagrange bilinear IFE solutions \cite{2008HeLinLin, 2012HeLinLin}. These plots are generated on the same mesh containing $80\times 80$ elements. We note that the error of bilinear IFE solution is much larger around interface than the rest of domain, in fact, the error at the interface looks like an ``interface crown". However, the nonconforming IFE solution is much more accurate than bilinear IFE solution around the interface. In fact, there is no apparent ``crown" around interface, which indicates the accuracy of nonconforming IFE solution around interface are comparable to the accuracy far away from the interface. 


\begin{figure}[!htb]
\caption{Comparison of point-wise errors of the nonconforming rotated $Q_1$ IFE and the bilinear IFE solutions}
\begin{center}
\includegraphics[width=0.49\textwidth]{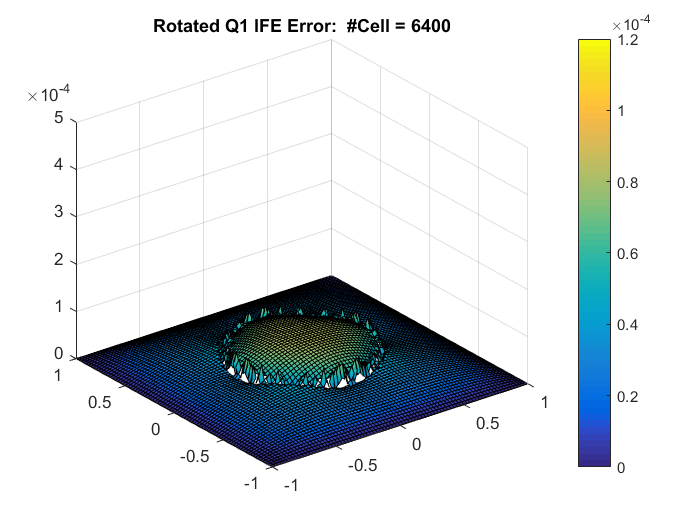}
\includegraphics[width=0.49\textwidth]{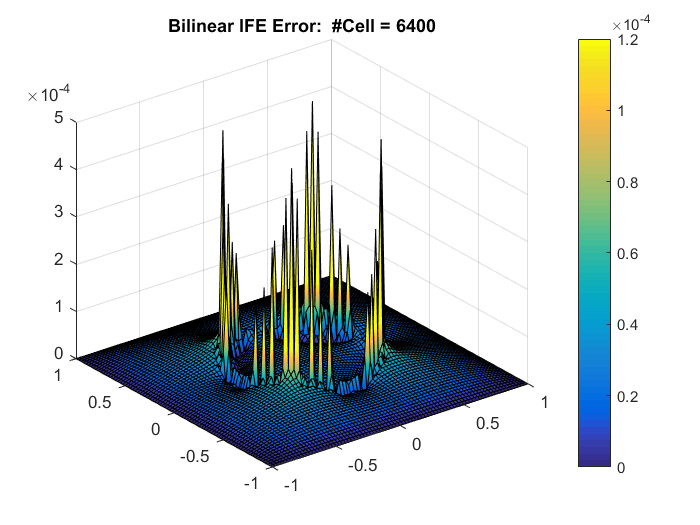}
\end{center}
\label{fig: comparison point-wise error}
\end{figure}

Next, we take a closer look of the nonconforming rotated-$Q_1$ IFE functions and Lagrange bilinear IFE functions \cite{2008HeLinLin}. In Figure \ref{fig: comparison Q1 RQ1 basis}, we plot the global bases of bilinear and rotated-$Q_1$ IFE function on two adjacent elements. In the left plot, there is a large gap on common interface edge of a bilinear IFE basis, where the continuity is only enforced at two endpoints of that edge.  To see it more clearly, in the right plot, the traces of this bilinear IFE function are plotted in blue curves which indicate that the largest discontinuity occurs at the intersection point of the edge and the interface. On the other hand, the middle plot shows that the discontinuity of a rotated-$Q_1$ IFE basis is scattered throughout the interface edge and it is less prominent. In fact, the rotated-$Q_1$ basis is weakly continuous across the edge in the sense that the mean values of the two traces are exactly the same.
The traces of this rotated-$Q_1$ IFE function are plotted in red curves which also demonstrate a smaller discontinuity across the interface edge. This shows why the rotated $Q_1$ IFE methods outperform bilinear IFE methods. 

\begin{figure}[tbh]
\begin{center}
\includegraphics[width=0.3\textwidth]{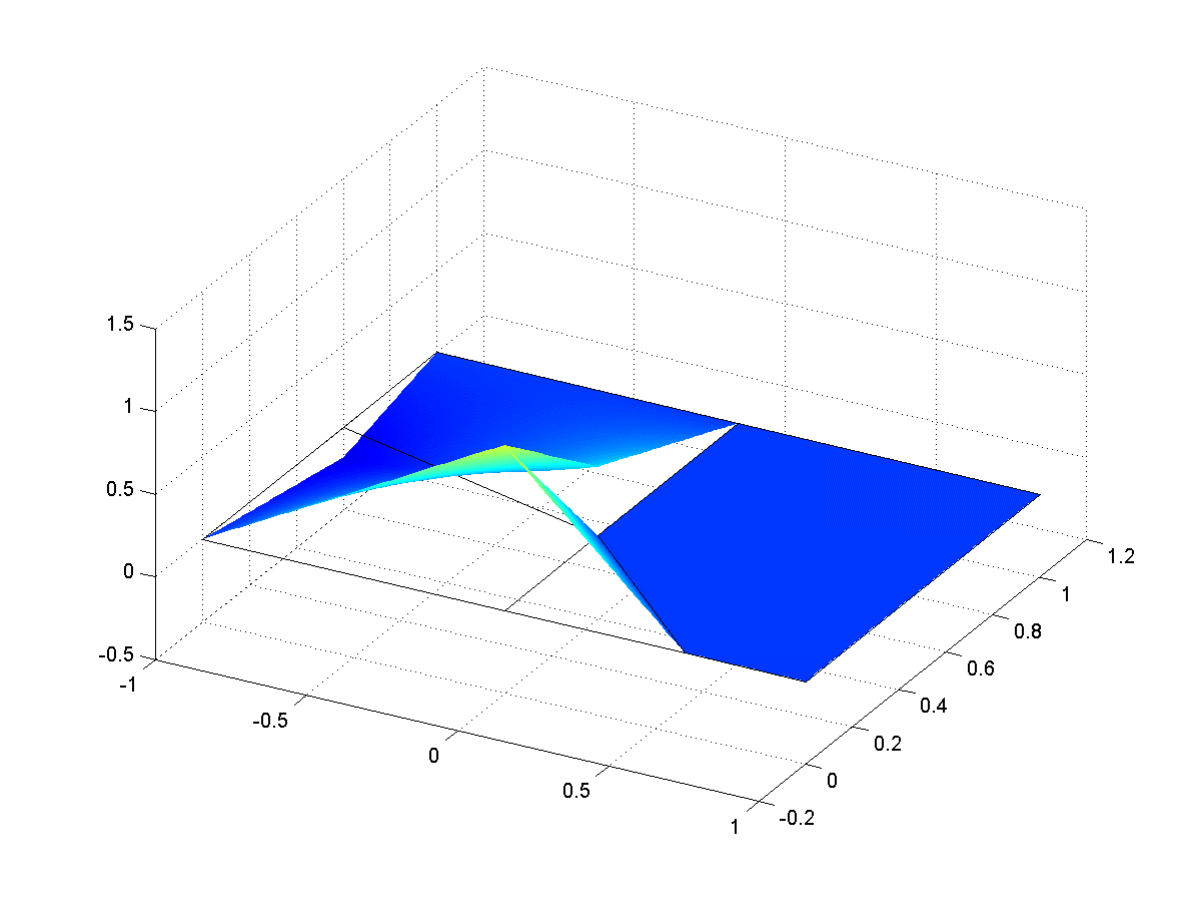}
\includegraphics[width=0.3\textwidth]{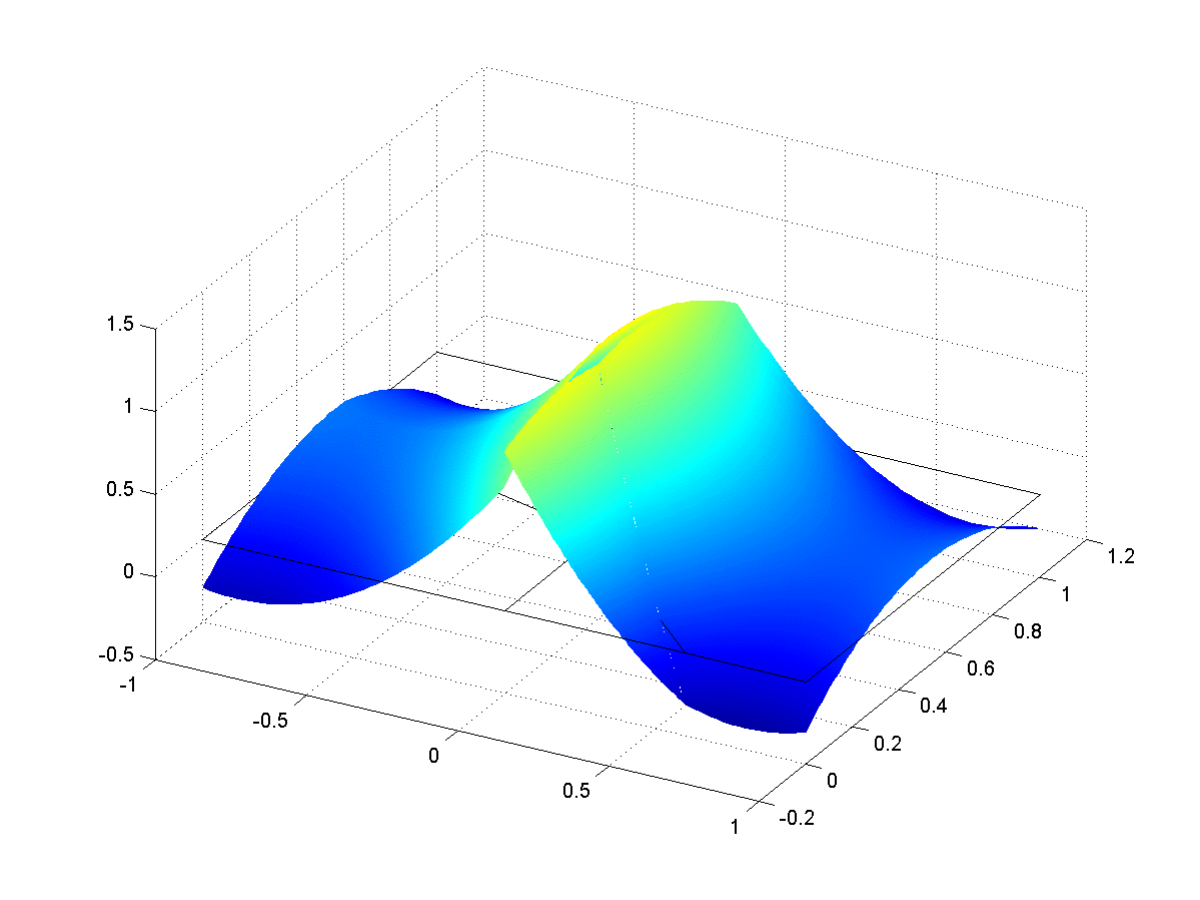}
\includegraphics[width=0.25\textwidth]{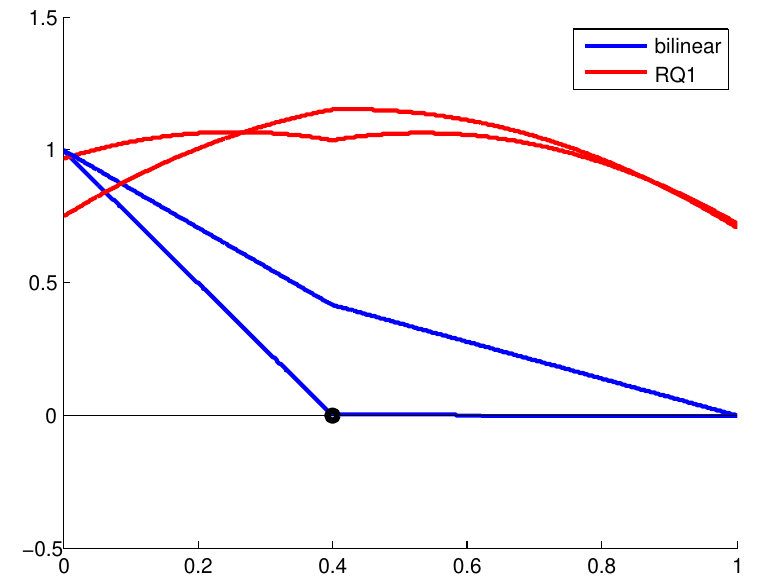}
\end{center}
\caption{Comparison of bilinear and rotated-$Q_1$ IFE global bases with $\beta^-=1$, $\beta^+ = 1000$.}
\label{fig: comparison Q1 RQ1 basis}
\end{figure}

As for robustness of our nonconforming IFE methods respect to the contrast of jumping coefficient, we test our methods in additional three coefficient configurations including the large contrast  $(\beta^-,\beta^+) = (1,10000)$, the flipping of the coefficients, i.e., $(\beta^-,\beta^+) = (10,1)$, and $(\beta^-,\beta^+) = (10000,1)$. 
Errors and convergence rates are reported in Tables \ref{table: Nonconforming IFE integral Galerkin integral 1 10000},  \ref{table: Nonconforming IFE integral Galerkin integral 10 1}, and \ref{table: Nonconforming IFE integral Galerkin integral 10000 1}, respectively. The convergence rates for all the cases are optimal, and this suggests the robustness of our IFE scheme with respect to the jump of coefficients.

\begin{table}[!hbt]
\caption{Errors of Galerkin IFE solutions $u-u_{h}$ with $\beta^-=1$, $\beta^+=10000$}
\begin{center}
\begin{tabular}{|c|cc|cc|cc|}
\hline
$N$ & $\|\cdot\|_{L^\infty}$ & rate & $\|\cdot\|_{L^2}$ & rate & $|\cdot|_{H^{1}}$ & rate \\
\hline
$10$      &$ 5.9646E{-3}$ & &$ 2.7360E{-3}$ & &$ 4.0678E{-2}$ &\\
$20$      &$ 2.5455E{-3}$ & 1.2285 &$ 1.0526E{-3}$ & 1.3782&$ 2.7824E{-2}$ & 0.5479\\
$40$      &$ 7.1692E{-4}$ & 1.8281 &$ 2.5767E{-4}$ & 2.0303&$ 1.4700E{-2}$ & 0.9205\\
$80$      &$ 2.1533E{-4}$ & 1.7353 &$ 6.3614E{-5}$ & 2.0181&$ 7.5491E{-3}$ & 0.9614\\
$160$     &$ 5.9653E{-5}$ & 1.8519 &$ 1.5531E{-5}$ & 2.0342&$ 3.7978E{-3}$ & 0.9911\\
$320$     &$ 1.5521E{-5}$ & 1.9423 &$ 4.0823E{-6}$ & 1.9277&$ 1.9146E{-3}$ & 0.9881\\
$640$     &$ 4.1575E{-6}$ & 1.9005 &$ 1.0069E{-6}$ & 2.0194&$ 9.5881E{-4}$ & 0.9977\\
$1280$    &$ 1.0588E{-6}$ & 1.9733 &$ 2.4921E{-7}$ & 2.0145&$ 4.8004E{-4}$ & 0.9981\\
\hline
\end{tabular}
\end{center}
\label{table: Nonconforming IFE integral Galerkin integral 1 10000}
\end{table}

\begin{table}[!hbt]
\caption{Errors of Galerkin IFE solutions $u-u_{h}$ with $\beta^-=10$, $\beta^+=1$}
\begin{center}
\begin{tabular}{|c|cc|cc|cc|}
\hline
$N$ & $\|\cdot\|_{L^\infty}$ & rate & $\|\cdot\|_{L^2}$ & rate & $|\cdot|_{H^{1}}$ & rate \\
\hline
$10$      &$2.5249E{-2}$ & &$1.0347E{-1}$ & &$ 1.8872E{-0}$ &\\
$20$      &$6.1647E{-3}$ &2.0341 &$2.6094E{-2}$ &1.9874 &$9.5266E{-1}$ & 0.9862\\
$40$      &$1.5899E{-3}$ &1.9551 &$6.5402E{-3}$ &1.9963 &$4.7745E{-1}$ & 0.9966\\
$80$      &$4.0597E{-4}$ &1.9695 &$1.6363E{-3}$ &1.9989 &$2.3887E{-1}$ & 0.9991\\
$160$     &$1.0382E{-4}$ &1.9673 &$4.0917E{-4}$ &1.9997 &$1.1945E{-1}$ & 0.9998\\
$320$     &$2.6221E{-5}$ &1.9853 &$1.0227E{-4}$ &2.0003 &$5.9730E{-2}$ & 0.9999\\
$640$     &$6.6006E{-6}$ &1.9900 &$2.5570E{-5}$ &1.9998 &$2.9865E{-2}$ & 1.0000\\
$1280$    &$1.6613E{-6}$ &1.9903 &$6.3931E{-6}$ &1.9999 &$1.4933E{-2}$ & 1.0000\\
\hline
\end{tabular}
\end{center}
\label{table: Nonconforming IFE integral Galerkin integral 10 1}
\end{table}

\begin{table}[!hbt]
\caption{Errors of Galerkin IFE solutions $u-u_{h}$ with $\beta^-=10000$, $\beta^+=1$}
\begin{center}
\begin{tabular}{|c|cc|cc|cc|}
\hline
$N$ & $\|\cdot\|_{L^\infty}$ & rate & $\|\cdot\|_{L^2}$ & rate & $|\cdot|_{H^{1}}$ & rate \\
\hline
$10$      &$2.5887E{-2}$ & &$1.0332E{-1}$ & &$ 1.8874E{-0}$ &\\
$20$      &$9.0928E{-3}$ &1.5094 &$2.6085E{-2}$ &1.9858 &$9.5275E{-1}$ & 0.9862\\
$40$      &$2.2570E{-3}$ &2.0144 &$6.5319E{-3}$ &1.9977 &$4.7747E{-1}$ & 0.9967\\
$80$      &$5.1846E{-4}$ &2.1180 &$1.6345E{-3}$ &1.9987 &$2.3887E{-1}$ & 0.9992\\
$160$     &$1.3253E{-4}$ &1.9679 &$4.0880E{-4}$ &1.9994 &$1.1945E{-1}$ & 0.9998\\
$320$     &$3.1459E{-5}$ &2.0748 &$1.0219E{-4}$ &2.0002 &$5.9729E{-2}$ & 0.9999\\
$640$     &$7.7833E{-6}$ &2.0150 &$2.5551E{-5}$ &1.9998 &$2.9865E{-2}$ & 1.0000\\
$1280$    &$1.9252E{-6}$ &2.1504 &$6.3885E{-6}$ &1.9998 &$1.4933E{-2}$ & 1.0000\\
\hline
\end{tabular}
\end{center}
\label{table: Nonconforming IFE integral Galerkin integral 10000 1}
\end{table}

\section{Conclusions}
In this article, we develop the rotated-$Q_1$ nonconforming IFE space based on integral value degrees of freedom. This new IFE space can be used in the usual Galerkin formulation to solve elliptic interface problems. The IFE space is proved to have optimal approximation capabilities. Error analysis of the Galerkin IFE solutions using integral-value degrees of freedom shows the quasi-optimal convergence rates in both energy and the $L^2$ norms.

\bibliographystyle{abbrv}
\bibliography{xuzhangBib}

\end{document}